\documentclass[11pt,reqno]{amsart}
\usepackage[margin=1in]{geometry}
\usepackage{amsmath,amssymb,amsthm,graphicx,amsxtra, setspace}
\usepackage[utf8]{inputenc}
\usepackage{mathrsfs}
\usepackage{hyperref}
\usepackage{upgreek}
\usepackage{mathtools}
\usepackage[dvipsnames]{xcolor}
\usepackage[mathcal]{euscript}
\usepackage{amsmath,units}
\usepackage{pgfplots}
\usepackage{tikz}
\usepackage{verbatim}
\allowdisplaybreaks
\usepackage{dsfont}
\usepackage{fontenc}
\usepackage{textcomp}
\usepackage{marvosym}
\usepackage{eurosym}
\usepackage{upgreek}
\usepackage[pagewise]{lineno}

\DeclareMathAlphabet{\mathpzc}{OT1}{pzc}{m}{it}

\usepackage[cyr]{aeguill}

\colorlet{darkblue}{blue!50!black}

\hypersetup{
	colorlinks,%
	citecolor=blue,%
	filecolor=red,%
	linkcolor=red,%
	urlcolor=blue,%
	pdfnewwindow=true,%
	pdfstartview={FitH}
}
\allowdisplaybreaks

\newtheorem{theorem}{Theorem}[section]

\newtheorem{definition}[theorem]{Definition}

\newtheorem{remark}[theorem]{Remark}

\renewcommand{\d}{\/\mathrm{d}\/}

\def\w{\textbf{W}^{\varepsilon}_{{\theta}^{\varepsilon}}}

\def\L{\mathbb{L}}
\def\wi{\widetilde}

\def\A{\mathrm{A}}

\def\C{\mathrm{C}}

\def\B{\mathrm{B}}

\def\E{\mathbb{E}}
\def\D{\mathrm{D}}

\def\y{\boldsymbol{y}}

\def\x{\boldsymbol{x}}

\def\bv{\boldsymbol{v}}
\def\v{\boldsymbol{v}}
\def\w{\boldsymbol{w}}
\def\W{\mathrm{W}}

\def\f{\boldsymbol{f}}
\def\V{\mathbb{V}}

\def\u{\boldsymbol{u}}

\def\H{\mathbb{H}}
\def\n{\boldsymbol{n}}

\newcommand{\R}{\mathbb{R}}

\renewcommand{\d}{\/\mathrm{d}\/}

\let\originalleft\left
\let\originalright\right
\renewcommand{\left}{\mathopen{}\mathclose\bgroup\originalleft}
\renewcommand{\right}{\aftergroup\egroup\originalright}

\newcommand{\vertiii}[1]{{\left\vert\kern-0.25ex\left\vert\kern-0.25ex\left\vert #1 
		\right\vert\kern-0.25ex\right\vert\kern-0.25ex\right\vert}}

\newcommand{\Addresses}{{
		\footnote{
			
			\noindent \textsuperscript{1}Department of Mathematics, Indian Institute of Technology Roorkee-IIT Roorkee,
			Haridwar Highway, Roorkee, Uttarakhand 247667, INDIA.\par\nopagebreak
			\noindent  \textit{e-mail:} \texttt{manilfma@iitr.ac.in, maniltmohan@gmail.com.}
			
			\noindent \textsuperscript{*}Corresponding author.

			\textit{Key words:} convective Brinkman-Forchheimer equations, global weak solutions, time-periodic solution, Brouwer’s fixed point theorem. 
			
			Mathematics Subject Classification (2010): 76D03, 35Q30, 35B10.

}}}

\begin{document}

	\title[Existence and uniqueness  of time-periodic solutions of 2D and 3D  CBFeD equations]{Existence and uniqueness  of time-periodic solutions of the 2D and 3D convective Brinkman-Forchheimer extended Darcy equations\Addresses}
	\author[M. T. Mohan]
	{Manil T. Mohan\textsuperscript{1*}}

	\maketitle
	
	\begin{abstract}
		 In this work, we investigate the existence and uniqueness of solutions to the following 2D and 3D convective Brinkman-Forchheimer extended Darcy equations defined on a bounded smooth domain $\Omega\subset\mathbb{R}^d$, $d\in\{2,3\}$,
		 	\begin{align*}
		 	\frac{\partial\boldsymbol{v}}{\partial t}-\mu \Delta\boldsymbol{v}+(\boldsymbol{v}\cdot\nabla)\boldsymbol{v}+\alpha\boldsymbol{v}+\beta\vert \boldsymbol{v}\vert^{r-1}\boldsymbol{v}+\gamma\vert \boldsymbol{v}\vert ^{q-1}\boldsymbol{v}+\nabla p=\boldsymbol{g},\ \nabla\cdot\boldsymbol{v}=0, 
		 \end{align*}
		 where $\mu,\alpha,\beta>0$, $\gamma\in\mathbb{R}$, $r,q\in[1,\infty)$ with $r>q\geq 1$ and $\boldsymbol{g}$ is an external forcing term. For $r \geq 1 $, under periodic forcing, we establish the \emph{existence of time-periodic global weak solutions} to the system by employing \emph{Faedo-Galerkin approximations}, together with the \emph{Banach-Alaoglu theorem}, the \emph{Aubin-Lions-Simon compactness lemma}, and the \emph{Lions-Magenes lemma}. The \emph{existence of periodic solutions} for the Faedo-Galerkin approximated problem is obtained via \emph{Brouwer’s fixed point theorem}. In the \emph{supercritical} case $( r > 3 )$ and the \emph{critical} case ($ r = 3$),  we prove the \emph{uniqueness of the global weak solution} without imposing any smallness condition on the external forcing. This constitutes a new result compared to the classical 2D  Navier-Stokes equations with periodic inputs, for which the \emph{uniqueness of strong solutions} typically requires \emph{smallness assumptions} on the external force.
	\end{abstract}

	
	\section{Introduction}\label{sec1}\setcounter{equation}{0}  
	
	Mathematical modeling and analysis of fluid dynamics are of fundamental interest, both for understanding fluid-related phenomena and for practical applications. In particular, the study of heat and fluid flow in porous media plays a crucial role in a wide range of scientific and engineering disciplines. Many models describing flow in porous media are based on \emph{Darcy’s law}, an empirical relation that puts forward a linear connection between the flow rate and the pressure gradient:
	$$\bv_d = -\frac{k}{\nu} \nabla p,$$
	where $\bv_d$ is the Darcy velocity, $k$ denotes the permeability, $\nu>0$ is the dynamic viscosity, and $p$ is the pressure of the fluid (cf. \cite{PAM}). However, numerous studies show that at higher flow rates, the linear relation predicted by Darcy’s law no longer holds. For example, in applications such as petroleum reservoirs and radial flow patterns, the flow exhibits nonlinear behavior. To account for this, \emph{Forchheimer} introduced a quadratic correction to the Darcy law, leading to the \emph{Darcy-Forchheimer law}:
	$$
	\nabla p = -\frac{\nu}{k}\bv_f - \gamma \rho_f |\bv_f| \bv_f,
	$$
	where $\bv_f$ is the Forchheimer velocity, $\gamma>0$ is the Forchheimer coefficient, and $\rho_f$ is the fluid density (cf. \cite{PAM}). Building on the Darcy-Forchheimer law, the \emph{convective Brinkman-Forchheimer extended Darcy model} was derived to describe fluid flow with thermal dispersion in porous media. This formulation arises naturally through the method of \emph{volume averaging}, which considers the deviations of velocity and temperature within the pore spaces:
		\begin{equation*}
			\left\{
			\begin{aligned}
				\frac{\partial \bv_f}{\partial t}- \mu\Delta\bv_f+(\bv_f\cdot\nabla)\bv_f+a_0\bv_f+a_1|\bv_f|\bv_f+\nabla p&=\boldsymbol{g}, \\ {\nabla\cdot\bv_f}&=0.
			\end{aligned}
			\right.
		\end{equation*}	
	From a mathematical perspective, the quadratic nonlinearity in the Forchheimer equation can be further generalized to include higher-order nonlinear terms. In particular, at high flow rates through porous media, the most relevant and practical case is captured by the \emph{linear-cubic Darcy-Forchheimer law} (cf. \cite{PAM}):
	$$
	\nabla p = -\frac{\nu}{k} \bv_f - \gamma \rho_f |\bv_f|^2 \bv_f.
$$
	Incorporating these nonlinear corrections to Darcy’s law,  the resulting \emph{convective Brinkman-Forchheimer extended Darcy model} reads as
	\begin{equation*}
		\left\{
		\begin{aligned}
			\frac{\partial \bv_f}{\partial t}-\mu\Delta\bv_f+ (\bv_f\cdot\nabla)\bv_f+a_0\bv_f+a_1|\bv_f|\bv_f+a_2|\bv_f|^2\bv_f+\nabla p&=\boldsymbol{g}, \\ \nabla\cdot\bv_f&=0.
		\end{aligned}
		\right.
	\end{equation*}	
	In this work, we study a further generalization of the convective Brinkman-Forchheimer extended Darcy equations by including a \emph{pumping term} that exhibits a similar nonlinear structure but with the opposite sign. 
	\subsection{The model}
Let $\Omega\subset\mathbb{R}^d$, $d\in\{2,3\}$ be a bounded domain 	with smooth boundary $\partial\Omega$ and let $T>0$ be fixed. The convective Brinkman-Forchheimer extended Darcy (CBFeD) equations model the motion of incompressible fluid flow in a saturated porous medium and are expressed as follows:
	\begin{equation}\label{eqn-model}
		\left\{
		\begin{aligned}
			\frac{\partial \bv}{\partial t}-\mu\Delta\bv+(\bv\cdot\nabla)\bv+\alpha\bv+\beta|\bv|^{r-1}\bv +\gamma|\bv|^{q-1}\bv+\nabla p&=\boldsymbol{g}, \ \text{ in } \ \Omega\times(0,T), \\ \nabla\cdot\bv&=0, \ \text{ in } \ \Omega\times[0,T], \\
			\bv&=\boldsymbol{0}\ \text{ in }\ \partial\Omega\times[0,T],\\
			\bv(0)&=\bv_0 \ \text{ in } \ \Omega,
		\end{aligned}
		\right.
	\end{equation}	
	where $\bv(x,t):\Omega\times[0,T]\to\R^d$ represents the velocity field at time $t$ and position $x$, $p(x,t):\Omega\times[0,T]\to\R$ denotes the pressure,  $\boldsymbol{g}(x,t):\Omega\times\mathbb{R}\to\R^d$ is an external  forcing term such that $\boldsymbol{g}(x,t+T)=\boldsymbol{g}(x,t)$ for all $t\in\mathbb{R}$. 
	The constant $\mu > 0$ represents the \emph{Brinkman coefficient} (effective viscosity), while the constants $\alpha$ and $\beta$ arise from the Darcy-Forchheimer law and are referred to as the \emph{Darcy coefficient} (related to the permeability of the porous medium) and the \emph{Forchheimer coefficient} (proportional to the porosity of the material), respectively. The nonlinear term $\gamma |\boldsymbol{v}|^{q-1} \boldsymbol{v}$ in \eqref{eqn-model} acts as \emph{damping} when $\gamma > 0$ and as \emph{pumping} when $\gamma < 0$. The parameter $r \in [1, \infty)$ is known as the \emph{absorption exponent}, with $r = 3$ corresponding to the \emph{critical exponent}. For $\gamma = 0$, the system reduces to the \emph{convective Brinkman-Forchheimer (CBF) equations} (\cite{KWH}). The critical homogeneous CBF equations (i.e., \eqref{eqn-model} with $r = 3$ and $\boldsymbol{g} = \mathbf{0}$ exhibit the same scaling as the \emph{Navier-Stokes equations (NSE)} only when $\alpha = 0$ (\cite{KWH}).
	The case $r < 3$ is referred to as \emph{subcritical}, whereas $r > 3$ corresponds to \emph{supercritical} or fast-growing nonlinearities. This model is particularly relevant when the flow velocity is too large for Darcy’s law to remain valid and when the porosity is not too small (\cite{PAM}). If one sets $\alpha = \beta = \gamma = 0$, the system reduces to the \emph{classical NSE}, and if $\alpha, \beta, \gamma > 0$, it can be interpreted as a \emph{damped NSE}.

	\subsection{Literature survey}
	The problem described in \eqref{eqn-model} is of considerable interest, not only because of its physical relevance but also from a mathematical standpoint. According to the classical results of \emph{Leray and Hopf}, the three-dimensional Navier-Stokes equations (NSE) admit at least one weak solution that satisfies the energy inequality. However, the \emph{uniqueness} of such weak solutions and the global existence of strong solutions remain one of the central open questions in mathematical fluid mechanics. To overcome these difficulties, several researchers have proposed various modifications of the 3D NSE (see, e.g., \cite{ZCQJ,KT2,ZZXW}). In particular, the authors of \cite{SNA1,SNA} introduced a modified version of the NSE incorporating an \emph{absorption term} of the form $\beta |\bv|^{r-1}\bv$ with $r \ge 1$, and established the \emph{existence of weak solutions} in any spatial dimension $d \ge 2$, along with \emph{uniqueness} results in two dimensions. More generally, equations that extend the classical NSE by including a \emph{damping term} of the form $\alpha \bv + \beta |\bv|^{r-1}\bv$ are commonly known as the \emph{convective Brinkman-Forchheimer (CBF) equations} (cf. \cite{KWH}). Furthermore, in \cite{PAM}, the authors examined a variant of the 3D NSE featuring both a \emph{damping} term $\beta |\bv|^{r-1}\bv$ and a \emph{pumping} term $\gamma |\bv|^{q-1}\bv$ (excluding the linear damping $\alpha \bv$), proving the \emph{existence of weak solutions} for $r > q$ and \emph{uniqueness} for $r > 3$. Analogous to the classical 3D NSE, the global existence of strong solutions to  the three-dimensional CBFeD equations remains open for $r \in [1,3)$ (for any $\beta, \mu > 0$) and for $r = 3$ when {$2\beta\mu < 1$}.

	The paper \cite{PAM} introduced and analyzed a continuous data assimilation algorithm for the three-dimensional Brinkman-Forchheimer-extended Darcy (3D BFeD) model, establishing the existence and uniqueness of solutions and proving the convergence of the proposed feedback-based assimilation scheme. The Brinkman-Forchheimer equations with rapidly growing nonlinearities were investigated in \cite{KT2}, where the authors demonstrated the existence of regular dissipative solutions and global attractors for the system \eqref{eqn-model} with $\gamma = 0$ in three dimensions for $r > 3$. This result ensures the existence of global weak solutions to the three-dimensional convective Brinkman-Forchheimer (CBF) equations, in the Leray-Hopf sense, satisfying the energy equality. The corresponding result for the critical case $(r = 3)$ was obtained in \cite{CLF}, where it was shown that all weak solutions of the critical CBF equations with $4\beta\mu \ge 1$, posed on a bounded domain in $\mathbb{R}^3$, satisfy the energy equality. In \cite{MTM7}, the monotonicity and hemicontinuity properties of the linear and nonlinear operators were analyzed, and, by employing the classical Minty-Browder technique, the authors established the existence and uniqueness of a global weak solution, again in the Leray-Hopf sense, satisfying the energy equality for the three-dimensional CBF equations with $r \ge 3$ (and $2\beta\mu \ge 1$ when $r = 3$). Furthermore, \cite{SGKKMTM} investigated the exponential stabilization of the controlled CBFeD system on a $d$-dimensional torus, employing both finite- and infinite-dimensional feedback controllers based on the theory of $m$-accretive operators and control design principles.

	The Navier-Stokes equations subject to time-periodic  external inputs have been studied from several complementary angles in the literature: for moderate forcing one can construct time-periodic weak and strong solutions (and prove uniqueness under smallness conditions), while for more general forcing the focus shifts to long-time statistical behavior, attractors and regularity of the solution set. In particular, classical functional-analytic treatments show existence of time-periodic solutions under smallness or suitable decay assumptions on the periodic body force and on domain/exterior-flow geometry (existence/uniqueness results,  \cite{GPG-13,HK-97,MK-14,DL-90,HMo,TN-25,ANFM-25,RAJPN-00,GPr,RS-94-1}; periodic exterior-domain constructions \cite{GPAL-06,GPG-22,PMMP-99,RS-94}). Classic approaches employ Galerkin approximations, fixed-point theorems (e.g., via the Poincar\'e map or Rothe method) to show existence of a weak or strong solution that is periodic in time and shares the period of the forcing. For example, Lauerov\'a \cite{DL-90} used the Rothe and Galerkin methods to establish periodic solutions for the variational form of Navier-Stokes. Kato \cite{HK-97} established the existence of time-periodic solutions to the NSE under appropriate smallness and regularity conditions on the periodic external force. Morimoto \cite{HMo} studied time-periodic solutions of the NSE with nonhomogeneous boundary conditions, employing the Galerkin approximation method to establish existence results under appropriate assumptions. Kyed \cite{MK-14} proved in the whole space setting that for small time-periodic data there is a strong time-periodic solution and uniqueness in a large class of weak solutions, together with regularity properties. Similarly, in the exterior-domain (flow around a body) setting,  Galdi  and Silvestre \cite{GPAL-06} proved existence of time-periodic motions of a fluid when the body undergoes a rigid periodic motion and a matching time-periodic body‐force is present. In the rotational Navier-Stokes problem with Coriolis force, Kozono et. a.l. \cite{HKYMRT-14} proved unique existence of time‐periodic solutions under small forcing and their asymptotic stability. To the best of our knowledge, time-periodic problems for both the CBF and CBFeD equations have not yet been investigated in the existing literature.

%

	\subsection{Novelties of the work}

In this work, we investigate the \emph{global solvability} of the CBFeD equations with periodic inputs. For $r \ge 1 $ and $ \f \in \mathrm{L}^2(0,T;\V') $, we establish the \emph{existence of a global weak solution}
 (Definition \ref{weakd}) $$
\bv \in \mathrm{C}_w([0,T];\H) \cap \mathrm{L}^2(0,T;\V) \cap \mathrm{L}^{r+1}(0,T;\widetilde{\L}^{r+1}),
$$
with
$$
\partial_t \bv \in \mathrm{L}^{p}(0,T;\V') + \mathrm{L}^{\frac{r+1}{r}}(0,T;\widetilde{\L}^{\frac{r+1}{r}}),
\quad
p =
\begin{cases}
	\frac{d}{4}, & r \in [1,3),\\
	2, & r \in [3,\infty),
\end{cases}
$$
satisfying the \emph{energy inequality} given in \eqref{ener-inequality}. The proof is based on the \emph{Faedo-Galerkin approximation}, the \emph{Banach-Alaoglu theorem}, the \emph{Aubin-Lions-Simon compactness lemma}, and the \emph{Lions-Magenes lemma} (see Theorem~\ref{thm2.11}). The \emph{existence of periodic solutions} for the Faedo-Galerkin approximated problem \eqref{a6} is established using \emph{Brouwer’s fixed point theorem}.

For $ r \in [1,\infty) $ when $d = 2 $ and for $ r \in [3,\infty) $ when $d = 3 $, we prove that $\bv \in \mathrm{C}([0,T];\H) $ and that the \emph{energy equality} given in \eqref{energy-equality} holds. Moreover, for $ r \in (3,\infty) $, the \emph{uniqueness of weak solutions} is obtained under either of the conditions
$$
\mu\lambda_1 + 2\alpha > \zeta, \quad \text{or} \quad
\left(\mu - \frac{1}{\beta}\right)\lambda_1 + \alpha > \eta, \quad \beta\mu > 1,
$$
{where $\lambda_1$ is the first eigenvalue of the Stokes operator (se Subsection \ref{sub2.2})}, and $ \zeta $ and $ \eta $ are defined in \eqref{eqn-zeta} and \eqref{eqn-eta}, respectively. For the critical case $ r = 3 $, uniqueness holds if $ \beta\mu > 1 $ and
$$
\alpha + \left(\mu - \frac{1}{\beta}\right)\lambda_1 > \kappa,
\quad
\kappa = \left(|\gamma| q 2^{q-2}\right)^{\frac{2}{3-q}}
\left(\frac{3-q}{2}\right)
\left(\frac{q-1}{\beta}\right)^{\frac{q-1}{2}}
\left[1 + 2^{\frac{q-1}{2}}\right].
$$
Importantly, \emph{no smallness assumption on the forcing term} is required to establish the uniqueness of weak solutions, marking a \emph{significant improvement} over the classical results for the NSE, where uniqueness typically relies on smallness assumptions on the external force (cf. \cite[Theorem 5.1]{HMo}, \cite[Theorem 1.2]{HK-97}).

{The existence of global weak solutions to the 2D and 3D CBF equations was established in \cite[Theorem 4.1]{SNA}. For the super-critical case, the existence and  uniqueness of weak solutions and the validity of the energy equality were proved in \cite[Theorem 3.3]{KT2} and \cite[Theorem 3.5]{MTM7} (see also \cite[Theorem 5.4]{CLF} for the energy equality in the critical 3D CBF equations in bounded domains and \cite[Theorem 1.4]{KWH} for the space-periodic setting). Furthermore, the existence of weak solutions in both super- and sub-critical cases, as well as the uniqueness of weak solutions and the existence of strong solutions in the super-critical case for the 3D CBFeD system with space-periodic boundary conditions, were addressed in \cite[Theorems 1.1, 1.2]{PAM}. It is important to note that all the aforementioned works address the standard initial value problem and focus on regular (non-time-periodic) solutions. In contrast to \cite{PAM}, the present paper is devoted to the study of time-periodic solutions to the 2D and 3D CBFeD equations \eqref{eqn-model} posed on bounded domains. Our main objective is to establish the existence of time-periodic weak solutions and, in the super-critical case, to prove the uniqueness of such weak solutions.}

	\subsection{Organization of the paper}

	The remainder of the paper is organized as follows. In the next section, we introduce the functional framework required to establish the existence of a  weak solution for system \eqref{eqn-model}. In Section~\ref{sect3}, we discuss the global existence ($d=2,3$ and $r\in[1,\infty)$) and uniqueness ($d=2,3$ and $r\in(3,\infty)$) of a weak solution to the periodic problem \eqref{eqn-model} (Theorem~\ref{thm2.11}). For $d=2,3$ and $r\in[1,3]$, under a suitable smallness assumption on the data and applying the Banach fixed point theorem, the uniqueness of strong solutions is obtained in Appendix \ref{App}.

	\section{Mathematical Formulation}\label{sec2}\setcounter{equation}{0}

This section introduces the essential functional spaces and summarizes the key properties of the associated linear and nonlinear operators needed to establish the global solvability of system \eqref{eqn-model}. The functional framework used here largely follows the approach in \cite{MTM7,MTM8}.

\subsection{Function spaces}

Let $\C_0^{\infty}(\Omega;\R^d)$ denote the space of all infinitely differentiable, $\R^d$-valued functions with compact support in a domain $\Omega \subset \R^d$.  Let us define 
\begin{align*} 
	\mathcal{V}&:=\{\u\in\C_0^{\infty}(\Omega,\R^d):\nabla\cdot\u=0\},\\
	\mathbb{H}&:=\text{the closure of }\ \mathcal{V} \ \text{ in the Lebesgue space } \L^2(\Omega)=\mathrm{L}^2(\Omega;\R^d),\\
	\mathbb{V}&:=\text{the closure of }\ \mathcal{V} \ \text{ in the Sobolev space } \H_0^1(\Omega)=\mathrm{H}_0^1(\Omega;\R^d),\\
	\widetilde{\L}^{p}&:=\text{the closure of }\ \mathcal{V} \ \text{ in the Lebesgue space } \L^p(\Omega)=\mathrm{L}^p(\Omega;\R^d),
\end{align*}
for $p\in(2,\infty)$. Then under some smoothness assumptions on the boundary, we characterize the spaces $\H$, $\V$ and $\widetilde{\L}^p$ as 
$$
\H=\{\u\in\L^2(\Omega):\nabla\cdot\u=0,\u\cdot\boldsymbol{n}\big|_{\partial\Omega}=0\},$$ with norm  $\|\u\|_{\H}^2:=\int_{\Omega}|\u(x)|^2\d x,
$
where $\boldsymbol{n}$ is the outward normal to $\partial\Omega$ and $\u\cdot\boldsymbol{n}$ is understood in the sense of trace (\cite[Section 1.3, Chapter 1]{Te}),
$$
\V=\{\u\in\H_0^1(\Omega):\nabla\cdot\u=0\},$$  with norm $ \|\u\|_{\V}^2:=\int_{\Omega}|\nabla\u(x)|^2\d x,
$ ({due to the Poincar\'e inequality given in \eqref{poin}}) and $$\widetilde{\L}^p=\{\u\in\L^p(\Omega):\nabla\cdot\u=0, \u\cdot\boldsymbol{n}\big|_{\partial\Omega}=0\},$$ with norm $\|\u\|_{\widetilde{\L}^p}^p=\int_{\Omega}|\u(x)|^p\d x$, respectively.
Let $(\cdot,\cdot)$ denote the inner product in the Hilbert space $\H$ and $\langle \cdot,\cdot\rangle $ represent the induced duality between the spaces $\V$  and its dual $\V'$ as well as $\widetilde{\L}^p$ and its dual $\widetilde{\L}^{p'}$, where $\frac{1}{p}+\frac{1}{p'}=1$. Note that $\H$ can be identified with its dual $\H'$. From \cite[Subsection 2.1]{FKS}, we have that the sum space $\V'+\widetilde{\L}^{p'}$ is well defined and  is a Banach space with respect to the norm 
\begin{align*}
	\|\u\|_{\V'+\widetilde{\L}^{p'}}&:=\inf\{\|\u_1\|_{\V'}+\|\u_2\|_{\wi\L^{p'}}:\u=\u_1+\u_2, \y_1\in\V' \ \text{and} \ \y_2\in\wi\L^{p'}\}\nonumber\\&=
	\sup\left\{\frac{|\langle\u_1+\u_2,\f\rangle|}{\|\f\|_{\V\cap\widetilde{\L}^p}}:\boldsymbol{0}\neq\f\in\V\cap\widetilde{\L}^p\right\},
\end{align*}
where $\|\cdot\|_{\V\cap\widetilde{\L}^p}:=\max\{\|\cdot\|_{\V}, \|\cdot\|_{\wi\L^p}\}$ is a norm on the Banach space $\V\cap\widetilde{\L}^p$. Also the norm $\max\{\|\u\|_{\V}, \|\u\|_{\wi\L^p}\}$ is equivalent to the norms  $\|\u\|_{\V}+\|\u\|_{\widetilde{\L}^{p}}$ and $\sqrt{\|\u\|_{\V}^2+\|\u\|_{\widetilde{\L}^{p}}^2}$ on the space $\V\cap\widetilde{\L}^p$. Moreover, we have the continuous embeddings $$\V\cap\widetilde{\L}^p\hookrightarrow\V\hookrightarrow\H\cong\H'\hookrightarrow\V'\hookrightarrow\V'+\widetilde{\L}^{p'},$$ where the embedding $\V\hookrightarrow\H$ is compact. 
	\subsection{Linear operator}\label{sub2.2}
	It is well known (see, for instance, \cite{DFHM,MHe,HKTY}) that every vector field $\u \in \L^p(\Omega)$, with $1 < p < \infty$, admits a unique {\emph{Helmholtz}}  decomposition of the form $	\u = \v + \nabla q,$
	where $\v \in \L^p(\Omega)$ satisfies $\mathrm{div\ }\v = 0$ in the sense of distributions in $\Omega$ and $\v \cdot \n = 0$ on $\partial \Omega$ (in the sense of traces), while $q \in \mathrm{W}^{1,p}(\Omega)$.
	For smooth vector fields, this decomposition is orthogonal in $\L^2(\Omega)$. Since the decomposition $\u = \bv + \nabla q$ holds for all $\u \in \L^p(\Omega)$, we can define the \emph{Helmholtz projection operator}
$	\mathcal{P}_p : \L^p(\Omega) \to \widetilde{\L}^p(\Omega),$
$	\mathcal{P}_p\u := \bv,$
	which projects a vector field onto its divergence-free component.
Let us also introduce the space
$	\mathcal{G}_p(\Omega) := \{\nabla q : q \in \mathrm{W}^{1,p}(\Omega)\},$
	endowed with the norm $\|\nabla q\|_{\L^p}$.
	Then, by the above decomposition, we obtain the direct sum
	$\L^p(\Omega) = \widetilde{\L}^p(\Omega) \oplus \mathcal{G}_p(\Omega),$
	where $\widetilde{\L}^p(\Omega)$ denotes the subspace of divergence-free vector fields in $\L^p(\Omega)$. 
	From  \cite[Theorem 1.4]{CSHS}, we infer 
	\begin{align*}
		\|\nabla q\|_{\L^p}\leq C\|\u\|_{\L^p}, \ \|\v\|_{\L^q}\leq (C+1)\|\u\|_{\L^p}\ \text{ and }\ \|\nabla q\|_{\L^p}+\|\v\|_{\L^q}\leq (2C+1)\|\u\|_{\L^p},
	\end{align*}
	where $C=C(\Omega,p)>0$ is a constant such that \begin{align*}
		\|\nabla q\|_{\L^p}\leq C\sup_{0\neq\nabla\varphi\in \mathcal{G}_{p'}(\Omega)}\frac{|\langle\nabla q,\nabla\varphi\rangle|}{\|\nabla\varphi\|_{\L^{p'}}}, \ \text{ for all }\ \nabla q\in \mathcal{G}_{p}(\Omega),
	\end{align*} 
	with $\frac{1}{p}+\frac{1}{p'}=1$. Setting $\mathcal{P}_p\u:=\v$, we obtain a bounded linear operator $\mathcal{P}_p:\L^p(\Omega)\to\wi\L^{p}(\Omega)$ such that $\mathcal{P}_p^2=\mathcal{P}_p$ (projection). For $p=2$, $\mathcal{P}:=\mathcal{P}_2:\L^2(\Omega)\to\H$ is an orthogonal projection.   Since $\Omega$ is of class $\C^2$, from \cite[Remark 1.6, Chapter 1, pp. 18]{Te}, we also infer that $\mathcal{P}$ maps $\H^1(\Omega)$ into itself and is continuous for the norm of $\H^1(\Omega)$.	

	We define the Stokes operator $\A$ by
	\begin{align*}
	\A\u := -\mathcal{P}\Delta\u, \  \u \in \D(\A) = \V \cap \H^2(\Omega),
	\end{align*}
	where $\mathcal{P}$ denotes the Helmholtz-Leray projection. It is well known that $\A$ is a \emph{non-negative self-adjoint operator} on $\H$, satisfying
	$\D(\A^{1/2}) = \V,$
	and
	\begin{align*}
		\langle \A\u, \u \rangle = \|\u\|_{\V}^2,
		\  \text{for all } \u \in \V,
	\end{align*}
	which in particular implies $\|\A\u\|_{\V'} \le \|\u\|_{\V}$. 
	When $\Omega$ is bounded, the operator $\A$ is \emph{invertible}, and its inverse $\A^{-1}$ is \emph{bounded, self-adjoint, and compact} in $\H$. Hence, by the \emph{spectral theorem}, the spectrum of $\A$ consists of an infinite sequence of positive eigenvalues
	$0 < \lambda_1 \le \lambda_2 \le \cdots \le \lambda_k \le \cdots, \  \lambda_k \to \infty \text{ as } k \to \infty,$
	with corresponding eigenfunctions $\{\boldsymbol{w}_k\}_{k=1}^\infty$ forming a complete orthonormal basis of $\H$, such that
$	\A \boldsymbol{w}_k = \lambda_k \boldsymbol{w}_k, \  \text{ for all }\  k \in \mathbb{N}.$
	Every $\u \in \H$ can thus be expanded as
$	\u = \sum_{k=1}^\infty (\u, \boldsymbol{w}_k )\boldsymbol{w}_k,$ so that 
$	\A\u = \sum_{k=1}^\infty \lambda_k (\u, \boldsymbol{w}_k )\boldsymbol{w}_k.$
	Consequently, we obtain the \emph{Poincar\'e inequality}:
	\begin{align}\label{poin}
		\|\nabla \u\|_{\H}^2
		= \langle \A\u, \u \rangle
		= \sum_{k=1}^\infty \lambda_k |(\u, \boldsymbol{w}_k )|^2
		\ge \lambda_1 \sum_{k=1}^\infty |( \u, \boldsymbol{w}_k )|^2
		= \lambda_1 \|\u\|_{\H}^2.
	\end{align}

	\subsection{Bilinear operator}
	
	Let us define the trilinear form $b(\cdot,\cdot,\cdot):\V\times\V\times\V\to\R$ by $$b(\u,\v,\w)=\int_{\Omega}(\u(x)\cdot\nabla)\v(x)\cdot\w(x)\d x=\sum\limits_{i,j=1}^2\int_{\Omega}u_i(x)\frac{\partial v_j(x)}{\partial x_i}w_j(x)\d x.$$ If $\u, \v$ are such that the linear map $b(\u, \v, \cdot) $ is continuous on $\V$, the corresponding element of $\V'$ is denoted by $\B(\u, \v)$. We also denote  $\B(\u) = \B(\u, \u)=\mathcal{P}[(\u\cdot\nabla)\u]$.
	An integration by parts yields 
	\begin{equation*}
		\left\{
		\begin{aligned}
			b(\u,\v,\v) &= 0,\ \text{ for all }\ \u,\v \in\V,\\
			b(\u,\v,\w) &=  -b(\u,\w,\v),\ \text{ for all }\ \u,\v,\w\in \V.
		\end{aligned}
		\right.\end{equation*}
	Using H\"older's inequality, we deduce the following inequality:
	\begin{align*}
		|b(\u,\v,\w)|=|b(\u,\w,\v)|\leq \|\u\|_{\wi\L^4}\|\nabla\w\|_{\H}\|\v\|_{\wi\L^4},\ \text{ for all }\ \u,\v,\w\in\V,
	\end{align*}
	and hence by using the Ladyzhenskaya inequality,  we get 
	\begin{align*}
		\|\mathrm{B}(\u,\v)\|_{\V'}\leq \|\u\|_{\wi\L^4}\|\v\|_{\wi\L^4}\leq \sqrt{2}\|\u\|_{\H}^{1/2}\|\u\|_{\V}^{1/2}\|\v\|_{\H}^{1/2}\|\v\|_{\V}^{1/2}, \ \text{ for all} \ \u,\v\in\V. 
	\end{align*}
	Furthermore, for $\u,\wi\u,\v\in\V$, an application of the Poincar\'e inequality (see \eqref{poin}) yields 
	\begin{align}\label{2b7}
		|\langle\B(\u)-\B(\widetilde{\u}),\v\rangle|&
		\leq\sqrt{\frac{2}{\lambda_1}}\left(\|\u\|_{\V}+\|\wi\u\|_{\V}\right)\|\u-\wi\u\|_{\V}\|\v\|_{\V}. 
	\end{align} Therefore, the map $\B : \V \to \V'$ is locally Lipschitz continuous. By using interpolation inequality, it can be shown  that $\B$ maps $ \V\cap\widetilde{\L}^{r+1}$  into $\V'+\widetilde{\L}^{\frac{r+1}{r}}$ since 
	\begin{align*}
	\left|\langle \B(\u,\u),\v\rangle \right|=\left|b(\u,\v,\u)\right|\leq \|\u\|_{\widetilde{\L}^{r+1}}\|\u\|_{\widetilde{\L}^{\frac{2(r+1)}{r-1}}}\|\v\|_{\V}\leq\|\u\|_{\widetilde{\L}^{r+1}}^{\frac{r+1}{r-1}}\|\u\|_{\H}^{\frac{r-3}{r-1}}\|\v\|_{\V},
	\end{align*}
	for all $\v\in\V\cap\widetilde{\L}^{r+1}$. Therefore, we deduce 
	\begin{align}\label{be-est}
	\|\B(\u)\|_{\V'+\widetilde{\L}^{\frac{r+1}{r}}}\leq\|\u\|_{\widetilde{\L}^{r+1}}^{\frac{r+1}{r-1}}\|\u\|_{\H}^{\frac{r-3}{r-1}}.
	\end{align}
	\subsection{Nonlinear operator}\label{sub2.4}
	We now consider the operator $\mathcal{C}(\u):=\mathcal{P}(|\u|^{r-1}\u)$ for $\u\in\V\cap\wi\L^{r+1}$. It is immediate that $\langle\mathcal{C}(\u),\u\rangle =\|\u\|_{\widetilde{\L}^{r+1}}^{r+1}$.
	Using Taylor's formula, we have (cf. \cite{MTM7})
	\begin{align}\label{213}
		|\langle \mathcal{C}(\u)-\mathcal{C}(\v),\w\rangle|&\leq r\left(\|\u\|_{\widetilde{\L}^{r+1}}+\|\v\|_{\widetilde{\L}^{r+1}}\right)^{r-1}\|\u-\v\|_{\widetilde{\L}^{r+1}}\|\w\|_{\widetilde{\L}^{r+1}},
	\end{align}
	for all $\u,\v,\w\in\V\cap\widetilde{\L}^{r+1}$. 
	Thus the operator $\mathcal{C}(\cdot):\V\cap\widetilde{\L}^{r+1}\to\V'+\widetilde{\L}^{\frac{r+1}{r}}$ is locally Lipschitz.
	Moreover, 	for any $r\in[1,\infty)$, one can establish that  (cf. \cite{MTM7})
	\begin{align}\label{2.23}
		\langle\mathcal{C}(\u)-\mathcal{C}(\v),\u-\v\rangle&\geq \frac{1}{2}\||\u|^{\frac{r-1}{2}}(\u-\v)\|_{\H}^2+\frac{1}{2}\||\v|^{\frac{r-1}{2}}(\u-\v)\|_{\H}^2\nonumber\\&\geq \frac{1}{2^{r-1}}\|\u-\v\|_{\wi\L^{r+1}}^{r+1} \geq 0,
	\end{align}
	for $r\geq 1$ (replace $2^{r-2}$ with $1,$ for $1\leq r\leq 2$). Similar properties hold true for $\widetilde{\mathcal{C}}(\u):=\mathcal{P}(|\u|^{q-1}\u)$ for $\u\in\V\cap\wi\L^{q+1}$.

The following result is a generalization of the well-known Lions-Magenes Lemma \cite{JLL+EM-I-72}, which will be used later to establish the energy equality. We denote the the space of distributions on $(0,T)$ with values in a Banach space $\mathbb{X}$ by $\mathcal{D}^\prime(0,T; \mathbb{X})$. 

\begin{theorem}[{\cite[Theorem 1.8]{VVC+MIV-02}}]\label{Thm-Abs-cont}
	Let $\H$ be a Hilbert space, and let $\V, \E, \mathbb{X}$ be Banach spaces, satisfying the inclusions
	\[\V \hookrightarrow \H \hookrightarrow \V^\prime \hookrightarrow \mathbb{X}\  \mbox{ and } \ \E \hookrightarrow \H \hookrightarrow \E^\prime \hookrightarrow \mathbb{X},\]
	where the spaces $\V^\prime$ and $\E^\prime$ are the duals of $\V$ and $\E$, respectively. Here the space $\H^\prime$ is identified with $\H$. Assume that $p>1$ and $\bv\in L^{2}(0,T;\V)\cap L^p (0,T; \E)$,  $\bv^{\prime}\in \mathcal{D}^\prime(0,T; \mathbb{X})$ and $\bv^\prime = \bv_1 + \bv_2$, where $\bv_1 \in L^{2}(0,T;\V^\prime)$ and $\bv_2 \in L^{p^\prime}(0,T;\E^{\prime})$. Then,
	\begin{itemize}
		\item[(i)] $\bv \in \mathrm{C}([0,T]; \H),$
		\item[(ii)] the function $[0,T] \ni t \mapsto\|\bv(t)\|_{\H}^2 \in \mathbb{R}$ is absolutely continuous on $[0,T]$, and
		\begin{align*}
			\|\bv(t)\|_{\H}^2 = \left\langle \bv(t),\bv^{\prime}(t)\right\rangle = 2\left\langle \bv(t),\bv_1(t)\right\rangle + 2\left\langle \bv(t),\bv_2(t)\right\rangle,
		\end{align*}
	\end{itemize}
	for a.e.  $t\in [0,T]$, that is, 
	\begin{align*}
		\|\bv(t)\|_{\H}^2 = \|\bv(t)\|_{\H}^2 + 2\int_0^t \big[ \left\langle \bv(s),\bv_1(s)\right\rangle + \left\langle \bv(s),\bv_2(s)\right\rangle\big] ds,
	\end{align*}
	for all $t\in[0,T]$. 
\end{theorem}

	\section{CBFeD Equations with Periodic Inputs} \label{sect3}\setcounter{equation}{0} In this section, we study the CBFeD equations with periodic inputs and establish the existence and uniqueness of global weak solutions. By applying the Helmholtz orthogonal projection $\mathcal{P}$ to system \eqref{eqn-model}, we obtain the following projected form of the CBFeD equations with periodic inputs:
	\begin{equation}\label{222}
		\left\{
		\begin{aligned}
			\partial_t\bv(t)+\mu\A\bv(t)+\B(\bv(t))+\alpha\bv(t)+\beta\mathcal{C}(\bv(t))+\gamma\widetilde{\mathcal{C}}(\bv(t))&=\f(t),\ \text{ in }\ \V^{\prime}+\widetilde{\mathbb{L}}^{\frac{r+1}{r}}, \\
			\bv(0)&=\bv(T),
		\end{aligned}
		\right.
	\end{equation}
	for a.e. $t\in[0,T]$, where $\f=\mathcal{P}\boldsymbol{g}$. The following definition of weak solutions is motivated from \cite[Defintion 2.1]{galdi-1}  and \cite[Definition 1.1.1, Chapter V]{HSo}.

	\begin{definition}\label{weakd}
		For $r\geq 1$,	a function  $$\bv\in\mathrm{L}^{\infty}(0,T;\H)\cap\mathrm{L}^2(0,T;\V)\cap\mathrm{L}^{r+1}(0,T;\widetilde\L^{r+1})),$$  with $\partial_t\bv\in\mathrm{L}^{p}(0,T;\mathbb{V}')+\mathrm{L}^{\frac{r+1}{r}}(0,T;\widetilde\L^{\frac{r+1}{r}}),$ where
		\begin{align}\label{value-p}
			p=\left\{\begin{array}{cl}\frac{d}{4}&\text{ for } \ r\in[1,3),\\
			2&\text{ for }\ 3\in[3,\infty),\end{array}\right.
		\end{align}
		 is called a \emph{weak solution} to the system (\ref{222}), if for $\f\in\mathrm{L}^2(0,T;\V')$, $\bv(\cdot)$ satisfies: $\bv(0)=\bv(T)\in\H$ and for all $t\in(0,T]$
	{	\begin{align}\label{3.13}
	-&	\int_0^t\int_{\Omega}	\bv(x,s)\partial_t\boldsymbol{\phi}(x,s)\d x\d s+\mu\int_0^t\int_{\Omega}\nabla\bv(x,s):\nabla\boldsymbol{\phi}(x,s)\d x\d s\nonumber\\&\quad+\int_0^t\int_{\Omega}((\bv(x,s)\cdot\nabla)\bv(x,s))\cdot\boldsymbol{\phi}(x,s)\d x\d s\nonumber\\&\quad+\int_0^t\int_{\Omega}\left[\alpha\bv(x,s)+\beta|\bv(x,s)|^{r-1}\bv(x,s)+\gamma|\bv(x,s)|^{q-1}\bv(x,s)\right]\cdot\boldsymbol{\phi}(x,s)\d x\d s\nonumber\\&=\int_{\Omega}\boldsymbol{v}(x,0)\boldsymbol{\phi}(x,0)\d x- \int_{\Omega}\boldsymbol{v}(x,t)\boldsymbol{\phi}(x,t)\d x+ \int_0^t\langle\f(s),\boldsymbol{\phi}(s)\rangle\d s,
		\end{align}
		for all time-periodic test functions $\boldsymbol{\phi}\in C_0^{\infty}([0,T];\mathcal{V})$.}
	\end{definition}
	\begin{definition}
		A \emph{Leray-Hopf weak solution} of problem \eqref{222} with the periodic condition $\bv(0)=\bv(T)\in\H$ is a weak solution satisfying the following \emph{strong energy inequality}:
		\begin{align}\label{ener-inequality}
		&2\mu\int_0^T\|\bv(t)\|_{\V}^2\d t+2\alpha\int_0^T\|\bv(t)\|_{\H}^2\d t+2\beta\int_0^T\|\bv(t)\|_{\wi\L^{r+1}}^{r+1}\d t+2\gamma\int_0^T\|\bv(t)\|_{\wi\L^{q+1}}^{q+1}\d t\nonumber\\&\leq 2\int_0^T\langle\f(t),\bv(t)\rangle\d t.
		\end{align}
	\end{definition}
	
	\begin{remark}
		1. The regularity  $\bv\in\mathrm{L}^{\infty}(0,T;\H)\cap\mathrm{L}^2(0,T;\V)\cap\mathrm{L}^{r+1}(0,T;\widetilde\L^{r+1}))$ and  $\partial_t\bv\in\mathrm{L}^{p}(0,T;\mathbb{V}')+\mathrm{L}^{\frac{r+1}{r}}(0,T;\widetilde\L^{\frac{r+1}{r}})\hookrightarrow \mathrm{L}^{\min\left\{p,\frac{r+1}{r}\right\}}(0,T;\mathbb{V}'+\widetilde\L^{\frac{r+1}{r}}) $ imply $\bv\in\W^{1,\min\left\{p,\frac{r+1}{r}\right\}}(0,T;\mathbb{V}'+\widetilde\L^{\frac{r+1}{r}})\hookrightarrow\C([0,T];\mathbb{V}'+\widetilde\L^{\frac{r+1}{r}}) $ (\cite[Theorem 2, pp. 302]{LCE}). Since $\H$ is reflexive and the embedding $\H\hookrightarrow\mathbb{V}'+\widetilde\L^{\frac{r+1}{r}}$ is continuous,   therefore, by an application of \cite[Lemma 8.1, Chapter 3, pp. 275]{JLL+EM-I-72} (also see \cite[Proposition 1.7.1, Chapter 1, pp. 61]{PCAM}) yields $\bv\in\C_w([0,T];\H)$, where $\bv\in\C_w([0,T];\H)$ denotes the space of functions $\bv:[0,T]\to \H$ which are weakly continuous. That is, for all $\boldsymbol{\zeta}\in\H$, the scalar function $[0,T]\ni t\mapsto (\bv(t),\boldsymbol{\zeta})\in\R$ is continuous on $[0,T]$. Therefore, the first two terms on the right-hand side of \eqref{3.13} are well defined.

	2. 	By using the density of $\mathcal{V}\subset \mathbb{V}\cap\wi\L^{r+1}$, we infer from \cite[Lemmas 2.1 and 2.2]{galdi-1} that \eqref{3.13} is equivalent to the following formation: for all $t\in(0,T]$
		\begin{align*}
			&\mu\int_0^t(\nabla\bv(s),\nabla\boldsymbol{\phi})\d s+\int_0^t\langle (\bv(s)\cdot\nabla)\bv(s),\boldsymbol{\phi}\rangle\d s\nonumber\\&\quad+\int_0^t\left\langle\left[\alpha\bv(s)+\beta|\bv(s)|^{r-1}\bv(s)+\gamma|\bv(s)|^{q-1}\bv(s)\right],\boldsymbol{\phi}\right\rangle\d s\nonumber\\&=(\bv(0),\boldsymbol{\phi})-(\bv(t),\boldsymbol{\phi})+\int_0^t\langle\f(s),\boldsymbol{\phi}\rangle\d s,
		\end{align*}
		for all $\boldsymbol{\phi}\in\V\cap\wi\L^{r+1}$. 
	\end{remark}
	
	We now state and prove the following result concerning the existence and uniqueness of weak solutions to system \eqref{222}. The existence of a weak solution is established by following the approach presented in \cite{HMo}. The uniqueness of weak solutions is proven without imposing any smallness assumptions on the forcing term; instead, the conditions are imposed on the parameters appearing in \eqref{222}.

	\begin{theorem}\label{thm2.11}
		Let  $\f\in\mathrm{L}^2(0,T;\V^{\prime})$  be given.  Then  there exists a weak solution to problem (\ref{222}) satisfying $\bv\in\mathrm{C}_w([0,T];\H)\cap\mathrm{L}^2(0,T;\V)\cap\mathrm{L}^{r+1}(0,T;\widetilde{\L}^{r+1}),  $ with $\partial_t\bv\in \mathrm{L}^{p}(0,T;\V^{\prime})+\mathrm{L}^{\frac{r+1}{r}}(0,T;\widetilde{\L}^{\frac{r+1}{r}})$, where $p$ is given in \eqref{value-p},  the energy estimate 
		\begin{align}\label{p36}
			&	\|\bv(t)\|_{\H}^2+\mu\int_0^t\|\bv(s)\|_{\V}^2\d s+2\alpha\int_0^t\|\bv(s)\|_{\H}^2\d s+2\beta\int_0^t\|\bv(s)\|_{\wi\L^{r+1}}^{r+1}\d s\nonumber\\&\leq \left(\frac{1}{\mu\lambda_1T}+1\right)\left\{2\left(\frac{r-q}{r+1}\right)
			\left(\frac{2(q+1)}{\beta(r+1)}\right)^{\frac{r-q}{q+1}}|\Omega|T+\frac{1}{\mu}\int_0^T\|\f(t)\|_{\V^{\prime}}^2\d t\right\},
		\end{align}
	for all $t\in[0,T]$	and the \emph{energy inequality} \eqref{ener-inequality}. 

	 Furthermore, for $r\in[1,\infty)$ when $d=2$ and  for $r\in[3,\infty)$ when $d=3$, $\bv\in\mathrm{C}([0,T];\H) $ and  the following \emph{energy equality} is satisfied: 
		\begin{align}\label{energy-equality}
			&	\|\bv(t)\|_{\H}^2+2\mu\int_0^t\|\bv(s)\|_{\V}^2\d s+2\alpha\int_0^t\|\bv(s)\|_{\H}^2\d s+2\beta\int_0^t\|\bv(s)\|_{\wi\L^{r+1}}^{r+1}\d s+2\gamma \int_0^t\|\bv(s)\|_{\wi\L^{q+1}}^{q+1}\d s \nonumber\\&=\|\bv(0)\|_{\H}^2+\int_0^t\langle\f(s),\v(s)\rangle\d s,
		\end{align}
		for all $t\in[0,T]$. 	
		
		For $r\in(3,\infty)$, if $\mu\lambda_1+2\alpha>\zeta$, where 
		\begin{align}\label{eqn-zeta}
		\small{\zeta= 2\left\{\left(\frac{1}{2\mu}\right)^{\frac{r-1}{r-3}}\left(\frac{r-3}{r-1}\right)\left(\frac{8}{\beta(r-1)}\right)^{\frac{2}{r-3}}+ \left(|\gamma|q2^{q-2}\right)^{\frac{r-1}{r-q}}\left(\frac{r-q}{r-1}\right)\left(\frac{2(q-1)}{\beta(r-1)}\right)^{\frac{q-1}{r-1}}\left[1+2^{\frac{q-1}{r-1}}\right]\right\}}
		\end{align}
		 or  if  $\left(\mu-\frac{1}{\beta}\right)\lambda_1+\alpha>\eta$ and $\beta\mu>1$, where 
		\begin{align}\label{eqn-eta}
			\eta= \left\{\frac{\beta}{4}+ \left(|\gamma|q2^{q-2}\right)^{\frac{r-1}{r-q}}\left(\frac{r-q}{r-1}\right)\left(\frac{2(q-1)}{\beta(r-1)}\right)^{\frac{q-1}{r-1}}\left[1+2^{\frac{q-1}{r-1}}\right]\right\},
			\end{align}   then the weak solution is unique.  Moreover, for $r=3$, the weak solution is unique if  $\beta\mu>1$ and $\alpha+\left(\mu-\frac{1}{\beta}\right)\lambda_1>\kappa$, where $\kappa=\left(|\gamma|q2^{q-2}\right)^{\frac{2}{3-q}}\left(\frac{3-q}{2}\right)\left(\frac{q-1}{\beta}\right)^{\frac{q-1}{2}}\left[1+2^{\frac{q-1}{2}}\right]$. 
	\end{theorem}
	\begin{proof}
		\textbf{Step (1):} \emph{Existence of a weak solution:} 
		
		\emph{Part (i). Finite dimensional Cauchy problem:}	Let $\{\w_1, \w_2, \ldots, \w_m, \ldots\}$ be a complete orthonormal basis of $\H$ contained in $\V$; for instance, one may take the eigenfunctions of the Stokes operator. Define $\H_m$ as the $m$-dimensional subspace of $\H$ spanned by $\{\w_1, \ldots, \w_m\}$. It holds that $\H_m\subset\H_{m+1}\subset\V$. Denote by $\mathrm{P}_m$, the orthogonal projection from $\V^{\prime}$ onto $\H_m$, that is, $\mathrm{P}_m \x = \sum_{i=1}^m \langle \x, \w_i \rangle \w_i.$
		Since each element $\x \in \H$ naturally defines a functional $\x^* \in \H$ through
		$\langle \x^*, \bv \rangle = (\x, \bv), \ \text{ for all }\  \bv \in \V,$
		the restriction of $\mathrm{P}_m$ to $\H$ coincides with the standard orthogonal projection of $\H$ onto $\H_m$, given by $\mathrm{P}_m \x = \sum_{i=1}^m (\x, \w_i) \w_i.$ In particular, $\mathrm{P}_m$ acts as the orthogonal projection from $\H$ onto the subspace $\operatorname{span}\{\w_1, \ldots, \w_m\}$.
For $T\in(0,\infty)$ and for each $m\in\mathbb{N}$, we seek an approximate solution of the form
		$\bv_m(t) = \sum_{k=1}^m g_{km}(t) \w_k,$
		which satisfies the initial value problem corresponding to the following finite-dimensional system of ordinary differential equations (ODEs):
		\begin{equation}\label{a1}
			\left\{
			\begin{aligned}
				&(\bv_m^{\prime},\w_j)+\mu(\nabla\bv_m,\nabla
			\w_j)+(\B(\bv_m),\w_j)+\alpha(\bv_m,\w_j)+\beta(\mathcal{C}(\bv_m),\w_j)\\&\quad+\gamma(\widetilde{\mathcal{C}}(\bv_m),\w_j)=\langle\f,\w_j\rangle, \\
				&	\bv_m(0)=\bv_0\in\text{span}\{\w_1,\ldots,\w_m\},
			\end{aligned}
			\right.
		\end{equation}
		for $1\leq j\leq m$, where $g_{km}:[0,T]\to\mathbb{R},$ $k\in\{1,\ldots,m\},$ $m\in\mathbb{N}$ are real-valued functions. Observe that $\mathrm{P}_m\bv_m=\bv_m$, hence, the preceding system of ODEs can equivalently be expressed in the form
		\begin{align*}
			\bv_m^{\prime}+\mu\A\bv_m+\mathrm{P}_m\B(\bv_m)+\alpha\bv_m+\beta\mathrm{P}_m\mathcal{C}(\bv_m)+\beta\mathrm{P}_m\widetilde{\mathcal{C}}(\bv_m)=\mathrm{P}_m\f. 
		\end{align*}
		Since the operators $\B(\cdot)$ and $\mathcal{C}(\cdot)$ are locally Lipschitz continuous (see \eqref{2b7} and \eqref{213}), \eqref{a1} is a system of non-linear, locally Lipschitz differential equations for the functions $g_{km}(\cdot)$. We can supplement this system of ODEs with initial conditions by setting for all $k$, $(g_{km}(0),\w_j) = (\bv_0,\w_j)$ for $1\leq j\leq m$. Carath\'eodory’s existence theorem guarantees the existence of a unique local solution $\bv_m(t)$ to system \eqref{a1} on some interval $t \in [0, t_m]$ with $\bv_m\in\mathrm{C}([0,t_m];\H_m)$. This proves the (local) existence of the function $\bv_m$. To show that this solution actually exists on the whole interval $[0, T]$, we  derive a uniform energy estimate for $\bv_m$. Multiplying equation \eqref{a1} by $g_{jm}(t)$ and summing over $1 \le j \le m$, we obtain
		\begin{align}\label{a2}
			&	\frac{1}{2}\frac{\d}{\d t}\|\bv_m(t)\|_{\H}^2+\mu\|\nabla\bv_m(t)\|_{\H}^2+\alpha\|\bv_m(t)\|_{\H}^2+\beta\|\bv_m(t)\|_{\wi\L^{r+1}}^{r+1}\nonumber\\&=-\kappa\|\bv_m(t)\|_{\wi\L^{q+1}}^{q+1}+\langle\f(t),\bv_m(t)\rangle\leq|\kappa||\Omega|^{\frac{r-q}{r+1}}\|\bv_m(t)\|_{\wi\L^{r+1}}^{q+1}+\|\f(t)\|_{\V^{\prime}}\|\nabla\bv_m(t)\|_{\H}\nonumber\\&\leq\frac{\beta}{2}\|\bv_m(t)\|_{\wi\L^{r+1}}^{r+1}+\left(\frac{r-q}{r+1}\right)
			\left(\frac{2(q+1)}{\beta(r+1)}\right)^{\frac{r-q}{q+1}}|\Omega|+ \frac{\mu}{2}\|\nabla\bv_m(t)\|_{\H}^2+\frac{1}{2\mu}\|\f(t)\|_{\V^{\prime}}^2,
		\end{align}
		for a.e. $t\in[0,T]$.	Integrating from $0$ to $t$, we find 
		\begin{align}\label{a4}
			&\|\bv_m(t)\|_{\H}^2+\mu\int_0^t\|\bv_m(s)\|_{\V}^2\d s+2\alpha\int_0^t\|\bv_m(s)\|_{\H}^2\d s+\beta\int_0^t\|\bv_m(s)\|_{\wi\L^{r+1}}^{r+1}\d s\nonumber\\&\leq\|\bv_0\|_{\H}^2+2\left(\frac{r-q}{r+1}\right)
			\left(\frac{2(q+1)}{\beta(r+1)}\right)^{\frac{r-q}{q+1}}|\Omega|t+\frac{1}{\mu}\int_0^t\|\f(s)\|_{\V^{\prime}}^2\d s,
		\end{align}
		for all $t\in[0,T]$. Since $\bv_0$ is the same initial vector for every $m$, observe that the right-hand side is a constant independent of $m$. Therefore, the maximal interval of existence can be extended, allowing us to conclude that $t_m = T$.
		Using the variation of constants formula in \eqref{a2}, we obtain 
		\begin{align}\label{a5}
			\|\bv_m(t)\|_{\H}^2&\leq e^{-(\lambda_1\mu+\alpha)t}\|\bv_0\|_{\H}^2+\frac{2}{(\lambda_1\mu+\alpha)}\left(\frac{r-q}{r+1}\right)
			\left(\frac{2(q+1)}{\beta(r+1)}\right)^{\frac{r-q}{q+1}}|\Omega|\nonumber\\&\quad+\frac{1}{\mu}\int_0^te^{-(\lambda_1\mu+\alpha)(t-s)}\|\f(s)\|_{\V^{\prime}}^2\d s,
		\end{align}
		 for all $ t\in[0,T]. $
		
		\emph{Part (ii). Finite-dimensional periodic problem:} 	Next, we consider the following finite dimensional periodic problem:
		\begin{equation}\label{a6}
			\left\{
			\begin{aligned}
				&(\bv_m^{\prime},\w_j)+\mu(\nabla\bv_m,\nabla \w_j)+(\B(\bv_m),\w_j)+\alpha(\bv_m,\w_j)+\beta(\mathcal{C}(\bv_m),\w_j)\\&\quad+\gamma(\widetilde{\mathcal{C}}(\bv_m),\w_j)=(\f,\w_j), \\
				&	\bv_m(0)=\bv_m(T),
			\end{aligned}
			\right.
		\end{equation}
		for $1\leq j\leq m$. From the preceding discussion, it follows that there exists a unique solution 
		$\bv_m(\cdot)$ to the initial value problem corresponding to the  initial condition $$\bv_m(0)=\bv_0\in\text{span}\{\w_1,\ldots,\w_m\}.$$ Let us now define a mapping $$\mathcal{T}_m:\text{span}\{\w_1,\ldots,\w_m\}\to\text{span}\{\w_1,\ldots,\w_m\}\ \text{ as }\ \mathcal{T}_m(\bv_0)=\bv_m(T).$$ Clearly, the mapping $\mathcal{T}_m$ is continuous from $\text{span}\{\w_1,\ldots,\w_m\}$ to $\text{span}\{\w_1,\ldots,\w_m\}$. Let us now define $\mathcal{B}_m(R)=\left\{\bv\in\text{span}\{\w_1,\ldots,\w_m\}:\|\bv\|_{\H}\leq R\right\}.$ Our next aim is to show that there exists a positive number $R$ independent of $m$ such that $\mathcal{T}_m(\mathcal{B}_m(R))\subset \mathcal{B}_m(R)$. Let us choose $R$ as $$R^2=\frac{\frac{1}{\mu}\int_0^Te^{-(\lambda_1\mu+\alpha)(T-t)}\|\f(t)\|_{\V^{\prime}}^2\d t+\frac{1}{(\lambda_1\mu+\alpha)}\left(\frac{r-q}{r+1}\right)
			\left(\frac{2(q+1)}{\beta(r+1)}\right)^{\frac{r-q}{q+1}}|\Omega|}{1-e^{-(\lambda_1\mu+\alpha)T}}.$$ Note that $R$ is independent of $m$, and if $\|\bv_0\|_{\H}\leq R$, from \eqref{a5}, we obtain  
		\begin{align}\label{eqn-310}
			\|\bv_m(T)\|_{\H}^2&=\|\bv_m(0)\|_{\H}^2=\|\bv_0\|_{\H}^2\nonumber\\&\leq e^{-(\lambda_1\mu+\alpha)T}\|\bv_0\|_{\H}^2+\frac{1}{(\lambda_1\mu+\alpha)}\left(\frac{r-q}{r+1}\right)
			\left(\frac{2(q+1)}{\beta(r+1)}\right)^{\frac{r-q}{q+1}}|\Omega|\nonumber\\&\quad+\frac{1}{\mu}\int_0^Te^{-(\lambda_1\mu+\alpha)(T-t)}\|\f(t)\|_{\V^{\prime}}^2\d t\nonumber\\&\leq e^{-(\lambda_1\mu+\alpha)T}R^2+R^2(1-e^{-(\lambda_1\mu+\alpha)T})=R^2. 
		\end{align}
		Thus, we have $\|\mathcal{T}_m\bv_0\|_{\H}=\|\bv_m(T)\|_{\H}\leq R$ and $\mathcal{T}_m(\mathcal{B}_m(R))\subset\mathcal{B}_m(R)$. Applying Brouwer's fixed point theorem (\cite[Theorem 3, Page 436]{LCE}), there exists $\bv_0\in\mathrm{span}\{\w_1,\ldots,\w_m\}$ such that $\mathcal{T}_m(\bv_0)=\bv_0$. Let $\bv_m(\cdot)$ be the solution of problem  \eqref{a1} with $\bv_m(0)=\bv_0$. Then $\bv_m(\cdot)$ is a periodic solution to \eqref{a6}. From \eqref{eqn-310}, we infer that  $\|\bv_m(0)\|_{\H}\leq R$, for all $m$.

			\emph{Part (iii). Uniform energy estimates:} 	
		Using a calculation similar to \eqref{a4} yields 
		\begin{align}\label{b1}
			&\|\bv_m(t)\|_{\H}^2+\mu\int_0^t\|\bv_m(s)\|_{\V}^2\d s+2\alpha\int_0^t\|\bv_m(s)\|_{\H}^2\d s+\beta\int_0^t\|\bv_m(s)\|_{\wi\L^{r+1}}^{r+1}\d s\nonumber\\&\leq\|\bv_m(0)\|_{\H}^2+2\left(\frac{r-q}{r+1}\right)
			\left(\frac{2(q+1)}{\beta(r+1)}\right)^{\frac{r-q}{q+1}}|\Omega|t+\frac{1}{\mu}\int_0^t\|\f(s)\|_{\V^{\prime}}^2\d s,
		\end{align}
		for all $t\in[0,T]$. The periodicity condition $\bv_m(0)=\bv_m(T)$ implies 
		\begin{align}\label{b2}
			&\mu\int_0^T\|\bv_m(t)\|_{\V}^2\d t+2\alpha\int_0^T\|\bv_m(t)\|_{\H}^2\d t+\beta\int_0^T\|\bv_m(t)\|_{\wi\L^{r+1}}^{r+1}\d t\nonumber\\&\leq 2\left(\frac{r-q}{r+1}\right)
			\left(\frac{2(q+1)}{\beta(r+1)}\right)^{\frac{r-q}{q+1}}|\Omega|T+\frac{1}{\mu}\int_0^T\|\f(t)\|_{\V^{\prime}}^2\d t=:\mathcal{K}.
		\end{align}
		Multiplying \eqref{a1} with $tg_{jm}(t)$ and summing up with respect to $1\leq j\leq m$ and then integrating it from $0$ to $t$, we find 
		\begin{align}\label{b3}
			&t\|\bv_m(t)\|_{\H}^2+2\mu\int_0^ts\|\bv_m(s)\|_{\V}^2\d s+2\alpha\int_0^ts\|\bv_m(s)\|_{\H}^2\d s+2\beta\int_0^ts\|\bv_m(s)\|_{\widetilde{\L}^{r+1}}^{r+1}\d s\nonumber\\&=\int_0^t\|\bv_m(s)\|_{\H}^2\d s+2\int_0^ts\langle\f(s),\bv_m(s)\rangle\d s-2\kappa\int_0^ts\|\bv_m(s)\|_{\widetilde{\L}^{q+1}}^{q+1}\d s\nonumber\\&\leq\frac{1}{\lambda_1}\int_0^T\|\bv_m(t)\|_{\V}^2\d t+\mu\int_0^ts\|\bv_m(s)\|_{\V}^2\d s+\frac{1}{\mu}\int_0^ts\|\f(s)\|_{\V^{\prime}}^2\d s
			\nonumber\\&\quad+\beta \int_0^ts\|\bv_m(s)\|_{\widetilde{\L}^{r+1}}^{r+1}\d s +\left(\frac{r-q}{r+1}\right)
			\left(\frac{2(q+1)}{\beta(r+1)}\right)^{\frac{r-q}{q+1}}|\Omega|t^2. 
		\end{align}
		Therefore, using  \eqref{b2} in  \eqref{b3}, we deduce 
		\begin{align*}
			t\|\bv_m(t)\|_{\H}^2\leq \left(\frac{1}{\mu\lambda_1}+T\right)\mathcal{K}.
		\end{align*}
	for all $t\in[0,T]$.	Since $\bv_m(0)=\bv_m(T)$, from the above estimate, we obtain 
		\begin{align*}
			\|\bv_m(0)\|_{\H}^2=\|\bv_m(T)\|_{\H}^2\leq \left(\frac{1}{\mu\lambda_1T}+1\right)
			\mathcal{K}.
		\end{align*}
		It is immediate from \eqref{b1} that 
		\begin{align}\label{b6}
			&\sup_{t\in[0,T]}\|\bv_m(t)\|_{\H}^2\leq \left(\frac{1}{\mu\lambda_1T}+1\right)
			\mathcal{K}. 
		\end{align}
	Along with  the estimate \eqref{b2}, we infer 
		\begin{align}\label{eqn-uni-bound}
			\{\bv_m\}_{m\in\mathbb{N}} \ \text{ is uniformly bounded in }\ \mathrm{L}^{\infty}(0,T;\H)\cap\mathrm{L}^2(0,T;\V)\cap\mathrm{L}^{r+1}(0,T;\wi\L^{r+1}).
		\end{align}
		Our next aim is to obtain a uniform estimate  on the time derivative.  Observe that for $1\leq d\leq 4$, Sobolev's embedding yields  $\D(\A)\hookrightarrow\V\cap\L^p$ for all $1\leq p<\infty$.  Let us fix any $\boldsymbol{\phi}\in\D(\A)$ with $\|\v\|_{\D(\A)}\leq 1$, and we write $\boldsymbol{\phi}=\boldsymbol{\phi}^1+\boldsymbol{\phi}^2$, where $\boldsymbol{\phi}^1\in\mathrm{span}\{\w_1,\ldots,\w_m\}$ and $(\boldsymbol{\phi}^2,\w_k)=0$, for $k=1,2,\ldots,m$. 	Since, $\{\w_j\}_{j\in\mathbb{N}}\subset\D(\A)$ are eigenfunctions of the Stokes operator, they are orthonormal  in $\H$ and $\|\boldsymbol{\phi}^1\|_{\D(\A)}\leq\|\boldsymbol{\phi}\|_{\D(\A)}\leq 1$. Using the Cauchy-Schwarz, H\"older, Ladyzhenskaya, Poincar\'e and Sobolev inequalities, we deduce from \eqref{a6}  that 
		\begin{align}\label{eqn-time-der}
			|\langle\bv_m^{\prime},\boldsymbol{\phi}\rangle|&=|(\bv_m^{\prime},\boldsymbol{\phi})|=|(\bv_m^{\prime},\boldsymbol{\phi}^1)|\nonumber\\&=|\langle\f,\boldsymbol{\phi}^1\rangle-[\mu(\nabla\bv_m,\nabla \boldsymbol{\phi}^1)+(\B(\bv_m),\boldsymbol{\phi}^1)+\alpha(\bv_m,\boldsymbol{\phi}^1)+\beta(\mathcal{C}(\bv_m),\boldsymbol{\phi}^1)\nonumber\\&\quad+\gamma(\widetilde{\mathcal{C}}(\bv_m),\boldsymbol{\phi}^1)]|
			\nonumber\\&\leq\|\f\|_{\V^{\prime}}\|\boldsymbol{\phi}^1\|_{\V}+\mu\|\nabla\bv_m\|_{\H}\|\nabla\boldsymbol{\phi}^1\|_{\H}+\|\bv_m\|_{\wi\L^4}^2\|\boldsymbol{\phi}^1\|_{\V}+\alpha\|\bv_m\|_{\H}\|\boldsymbol{\phi}^1\|_{\H}\nonumber\\&\quad+\beta\|\bv_m\|_{\wi\L^{r+1}}^{r}\|\boldsymbol{\phi}^1\|_{\wi\L^{r+1}}+|\gamma|\|\bv_m\|_{\wi\L^{q+1}}^q\|\boldsymbol{\phi}^1\|_{\wi\L^{q+1}}
			\nonumber\\&\leq\left(\frac{1}{\sqrt{\lambda_1}}\|\f\|_{\V^{\prime}}+\frac{\mu}{\sqrt{\lambda_1}}\|\bv_m\|_{\V}+\frac{\alpha}{\lambda_1}\|\bv_m\|_{\H}+2^{\frac{d}{2}}\|\v_m\|_{\H}^{\frac{4-d}{2}}\|\v_m\|_{\V}^{\frac{d}{2}}\right.\nonumber\\&\qquad\left.+C\beta\|\bv_m\|_{\wi\L^{r+1}}^{r}+C|\gamma||\Omega|^{\frac{r-q}{r+1}}\|\bv_m\|_{\wi\L^{r+1}}^{q}\right)\|\boldsymbol{\phi}^1\|_{\D(\A)},
		\end{align}
		where $\lambda_1$ is the first eigenvalue of the Dirichlet Laplacian. For $r\geq 1,$ when $d=2$ and  $r\geq 3,$ when $d=3$, we immediately have 
		\begin{align*}
			\int_0^T\|\bv_m^{\prime}(t)\|_{(\D(\A))^{\prime}}^{\frac{r+1}{r}}\d t&\leq \left(\frac{T^{\frac{r-1}{2r}}}{\lambda_1^{\frac{r+1}{2r}}}\|\f\|_{\mathrm{L}^2(0,T;\V^{\prime})}^{\frac{r+1}{r}}+\frac{\mu T^{\frac{r-1}{2r}}}{\lambda_1^{\frac{r+1}{2r}}}\|\bv_m\|_{\mathrm{L}^2(0,T;\V)}^{\frac{r+1}{r}}+\alpha^{\frac{r+1}{r}} T\|\bv_m\|_{\mathrm{L}^{\infty}(0,T;\H)}^{\frac{r+1}{r}}\right.\nonumber\\&\quad+\left.  2^{\frac{d(r+1)}{2r}} T^{\frac{(4-d)r-d}{4r}}\|\bv_m\|_{\mathrm{L}^{\infty}(0,T;\H)}^{\frac{(4-d)(r+1)}{2r}}\|\bv_m\|_{\mathrm{L}^2(0,T;\V)}^{\frac{d(r+1)}{2r}} +C\beta^{\frac{r+1}{r}}\|\bv_m\|_{\mathrm{L}^{r+1}(0,T;\wi\L^{r+1})}^{r+1}\right.\nonumber\\&\quad\left.+C|\gamma|^{\frac{r+1}{r}}|\Omega|^{\frac{r-q}{r}}T^{\frac{r-q}{r}}\|\bv_m\|_{\mathrm{L}^{r+1}(0,T;\wi\L^{r+1})}^{\frac{(r+1)q}{r}}\right),
		\end{align*}
		and the right hand side is bounded and independent of $m$ (see \eqref{b2} and \eqref{b6}).  For $r\in[1,3]$ when $d=3$, we consider 
		\begin{align*}
				\int_0^T\|\bv_m^{\prime}(t)\|_{(\D(\A))^{\prime}}^{\frac{4}{3}}\d t&\leq\left(\frac{T^{\frac{1}{3}}}{\lambda_1^{\frac{2}{3}}}\|\f\|_{\mathrm{L}^2(0,T;\V^{\prime})}^{\frac{4}{3}}+\frac{\mu T^{\frac{1}{3}}}{\lambda_1^{\frac{2}{3}}}\|\bv_m\|_{\mathrm{L}^2(0,T;\V)}^{\frac{4}{3}}+\alpha^{\frac{4}{3}} T\|\bv_m\|_{\mathrm{L}^{\infty}(0,T;\H)}^{\frac{4}{3}}\right.\nonumber\\&\quad\left. +4 \|\bv_m\|_{\mathrm{L}^{\infty}(0,T;\H)}^{\frac{2}{3}}\|\bv_m\|_{\mathrm{L}^2(0,T;\V)}+C\beta^{\frac{4}{3}} T^{\frac{3-r}{4r}} \|\bv_m\|_{\mathrm{L}^{r+1}(0,T;\wi\L^{r+1})}^{\frac{4r}{3}}\right.\nonumber\\&\quad+\left. C|\gamma|^{\frac{4}{3}} |\Omega|^{\frac{4(r-q)}{3(r+1)}} T^{\frac{3r+3-4q}{3(r+1)}} \|\bv_m\|_{\mathrm{L}^{r+1}(0,T;\wi\L^{r+1})}^{\frac{4q}{3}}\right),
		\end{align*}
		where also the right hand side is bounded and independent of $m$ (see \eqref{b2} and \eqref{b6}) .  
		
		 From the above discussions, one can conclude that 
		\begin{align}\label{eqn-uni-bound-1}
			\{\bv_m'\}_{m\in\mathbb{N}} \ \text{ is uniformly bounded in }\ \mathrm{L}^p(0,T;(\D(\A))^{\prime}),
		\end{align}
		where 
		\begin{align}\label{eqn-value-p}
			p=\left\{\begin{array}{cl}\frac{r+1}{r}, &\mbox{ for $r\geq 1,$ when $d=2$ and  $r\geq 3,$ when $d=3$},\\
			\frac{4}{3},&\mbox{ for $r\in[1,3],$ when $d=3$}.   \end{array}\right.
		\end{align}

			\emph{Part (iv). Convergence:} 	
		  Using \eqref{eqn-uni-bound} (see \eqref{b2} and \eqref{b6}) and \eqref{eqn-uni-bound-1},  application of the Banach-Alaoglu Theorem yields the existence of a subsequence still denoted by the same symbol such that 
		  	\begin{equation}\label{weak-convergence}
		  	\left\{
		  	\begin{aligned}
		  		&\bv_m\xrightarrow{w^*}\bv\ \text{ in }\ \mathrm{L}^{\infty}(0,T;\H),\\
		  		& \bv_m\xrightarrow{w}\bv\ \text{ in }\ \mathrm{L}^{2}(0,T;\V)\cap\mathrm{L}^{r+1}(0,T;\wi\L^{r+1}), 
		  	\\  &\bv_m'\xrightarrow{w}\partial_t\bv \ \text{ in }\ \mathrm{L}^p(0,T;(\D(\A))^{\prime}),
		  	\end{aligned}
		  	\right.
		  \end{equation}
	where $p$ is defined in \eqref{eqn-value-p}. Observe that 
	\begin{align*}
	\mbox{$\int_0^T\|\mathcal{C}(\bv_m(t)) \|_{\wi\L^{\frac{r+1}{r}}}^{\frac{r+1}{r}}\d t\leq C\int_0^T\|\bv_m(t)\|_{\wi\L^{r+1}}^{r+1}\d t$ and $\int_0^T\|\widetilde{\mathcal{C}}(\bv_m(t)) \|_{\wi\L^{\frac{q+1}{q}}}^{\frac{q+1}{q}}\d t\leq C\int_0^T\|\bv_m(t)\|_{\wi\L^{q+1}}^{q+1}\d t$}
	\end{align*} are bounded independent of $m$. Therefore, we have the following convergences: 
	\begin{align}\label{weak-convergence-1}
		&\mathcal{C}(\bv_m) \xrightarrow{w} \ \boldsymbol{\xi}\ \text{ in }\ \mathrm{L}^{\frac{r+1}{r}}(0,T;\wi\L^{\frac{r+1}{r}}), \quad \widetilde{\mathcal{C}}(\bv_m) \xrightarrow{w} \ \widetilde{\boldsymbol{\xi}}\ \text{ in }\ \mathrm{L}^{\frac{q+1}{q}}(0,T;\wi\L^{\frac{q+1}{q}}).
	\end{align}
	Since $\D(\A)\hookrightarrow\V\cap\wi\L^{r+1}\hookrightarrow\H\equiv\H^{\prime}\hookrightarrow\V^{\prime}+\wi\L^{\frac{r+1}{r}}\hookrightarrow (\D(\A))^{\prime}$, and the embedding of $\D(\A)\hookrightarrow\H$ is compact,   the {Aubin-Lions-Simon compactness Theorem (cf. \cite[Theorem 5]{SJ})} assures the existence of a subsequence of $\{\bv_m\}$ (still denoted by $\{\bv_m\}$) such that 
		\begin{equation}\label{eqn-strong-conv}
		\bv_m\to \bv\ \text{ in }\ \mathrm{L}^2(0,T;\H). 
		\end{equation}
		as $m\to\infty$. Moreover, along a further subsequence (using the Riesz-Fischer Theorem), we have 
		\begin{align}\label{eqn-point}
			\bv_m(t,x)\to \bv(t,x)\ \text{ for a.e. }\ (t,x)\in(0,T)\times \Omega.
		\end{align}
		Let us now take $\w\in \mathrm{C}([0,T];\H_n)$ for $n<m$ and consider
		\begin{align*}
			&\left|\int_0^Tb(\bv_m(t),\bv_m(t),\w(t))\d t-\int_0^Tb(\bv(t),\bv(t),\w(t))\d t\right|
			\nonumber\\&\leq\left|\int_0^Tb(\bv_m(t)-\bv(t),\bv_m(t),\w(t))\d t\right|+\left| \int_0^Tb(\bv(t),\bv_m(t)-\bv(t),\w(t))\d t\right|\nonumber\\&:=I_1+I_2. 
		\end{align*}
	The convergence \eqref{weak-convergence} implies that $I_2\to 0$ as $m\to\infty$ for all $n\in\mathbb{N}$.	Using H\"older's, Ladyzhneskaya's  and Young's inequalities, and the uniform bound \eqref{eqn-uni-bound} (see \eqref{b2} and \eqref{b6}),  and  the convergence  $\bv_m\to \bv\ \text{ in }\ \mathrm{L}^2(0,T;\H),$  we find 
		\begin{align*}
			I_1&=\left|\int_0^Tb(\bv_m(t)-\bv(t),\bv_m(t),\w(t))\d t\right|
			\nonumber\\&\leq\int_0^T\|\bv_m(t)-\bv(t)\|_{\wi\L^4}\|\bv_m(t)\|_{\V}\|\w(t)\|_{\wi\L^4}\d t\nonumber\\&\leq 2^{\frac{d}{2}}\sup_{t\in[0,T]}\|\w(t)\|_{\wi\L^4}\int_0^T\|\bv_m(t)-\bv(t)\|_{\H}^{1-\frac{d}{4}}\|\bv_m(t)-\bv(t)\|_{\V}^{\frac{d}{4}}\|\bv_m(t)\|_{\V}\d t\nonumber\\&\leq C\sup_{t\in[0,T]}\|\w(t)\|_{\wi\L^4}\left(\int_0^T\|\bv_m(t)-\bv(t)\|_{\H}^2\d t\right)^{\frac{d-4}{8}}\left(\int_0^T\left(\|\bv_m(t)\|_{\V}^2+\|\bv(t)\|_{\V}^{2}\right)\d t\right)^{\frac{4+d}{8}}\nonumber\\&\to 0\ \text{ as } \ m\to\infty, 
			\end{align*}
	for all $n\in\mathbb{N}$. 	Since $ \bv_m\xrightarrow{w}\bv\ \text{ in }\ \mathrm{L}^{2}(0,T;\V)$, one easily gets 
		\begin{align*}
		I_1=\left|	\int_0^Tb(\bv(t),\bv_m(t)-\bv(t),\w(t))\d t\right|\to 0\ \text{ as } m\to\infty,
		\end{align*}
		since $\bv\in\mathrm{L}^{\frac{8}{d}}(0,T;\wi\L^4)$ and it is true for all $n\in\mathbb{N}$.  Since the above convergence is independent of $n$ and  the embedding of $\mathrm{C}([0,T];\H_n)\subset \mathrm{L}^2(0,T;\V)$ is dense,  a density argument  yields the convergence 
		\begin{align*}
			\int_0^Tb(\bv_m(t),\bv_m(t),\boldsymbol{\varphi}(t))\d t\to \int_0^Tb(\bv(t),\bv(t),\boldsymbol{\varphi}(t))\d t,
		\end{align*}
		for all $\boldsymbol{\varphi}\in \mathrm{L}^2(0,T;\V)$.  Owing to the bound \eqref{b2}, the convergence \eqref{eqn-point} and {(\cite[Lemma 1.3]{lions-1})}, we obtain
		\begin{align}\label{absorption cgts}
			|\bv_m|^{r-1}\bv_m & \xrightarrow{w} |\bv|^{r-1}\bv \ \text{ in }  \ \mathrm{L}^{\frac{r+1}{r}}(0, T; \wi{\mathbb{L}}^{\frac{r +1}{r}}).
		\end{align}
	We need to show that $\boldsymbol{\xi}=\mathcal{P}( |\bv|^{r-1}\bv)=\mathcal{C}(\bv)$.	For all $\boldsymbol{\varphi}\in\mathrm{L}^{r+1}(0,T;\H_n)$ for $n<m$, we consider 
	\begin{align*}
&	\left|	\int_0^T\langle\mathcal{C}(\bv_m(t))-\mathcal{C}(\bv(t)),\boldsymbol{\varphi}(t)\rangle\d t\right|
\nonumber\\&= \left|	\int_0^T\langle 	|\bv_m(t)|^{r-1}\bv_m(t) -	|\bv(t)|^{r-1}\bv(t) ,\mathcal{P}\boldsymbol{\varphi}(t)\rangle\d t\right|\to 0\ \text{ as }\ m\to\infty, 
	\end{align*}
	by using the convergence given in \eqref{absorption cgts} and it holds true for all $n\in\mathbb{N}$. 
	Since the embedding of $\mathrm{C}([0,T];\H_n)\subset \mathrm{L}^2(0,T;\D(\A))$ is dense,  a density argument  yields the convergence 
	\begin{align*}
	\int_0^T\langle\mathcal{C}(\bv_m(t)),\boldsymbol{\varphi}(t)\rangle\d t\to \int_0^T\langle\mathcal{C}(\bv(t)),\boldsymbol{\varphi}(t)\rangle\d t\ \text{ as }\ m\to\infty, 
	\end{align*}
	for all $\boldsymbol{\varphi}\in \mathrm{L}^2(0,T;\D(\A))$. Since $\D(\A)\hookrightarrow \wi\L^{r+1}$ is dense, the same result holds true for all  $\boldsymbol{\varphi}\in \mathrm{L}^{r+1}(0,T;\wi\L^{r+1})$.   By the uniqueness of weak limits (see \eqref{weak-convergence-1}), we obtain $\mathcal{C}(\bv)=\boldsymbol{\xi}$. Let us now show that 
	\begin{align}\label{eqn-nonlinear-storng}
		\widetilde{\mathcal{C}}(\bv_m)\to 	\widetilde{\mathcal{C}}(\bv)  \ \text{ in }\ \mathrm{L}^{\frac{q+1}{q}}(0,T;\wi\L^{\frac{q+1}{q}})\ \text{ as }\ m\to\infty.
	\end{align}
In order to establish the above convergence, it is enough to show 
\begin{align}\label{nonlinear-storng}	|\bv_m|^{q-1}\bv_m \to |\bv|^{q-1}\bv \ \text{ in }  \ \mathrm{L}^{\frac{q+1}{q}}(0, T; \wi{\mathbb{L}}^{\frac{q+1}{q}}).\end{align} We consider 
	\begin{align*}
&\int_0^T\||\bv_m(t)|^{q-1}\bv_m(t)-|\bv(t)|^{q-1}\bv(t)\|_{\wi\L^{\frac{q+1}{q}}}^{\frac{q+1}{q}}\d t
		\nonumber\\&\leq  C\int_0^T\||\bv_m(t)|^{q-1}(\bv_m(t)-\bv(t))\|_{\wi\L^{\frac{q+1}{q}}}^{\frac{q+1}{q}}\d t+ C\int_0^T\|\left(|\bv_m(t)|^{q-1}-|\bv(t)|^{q-1}\right)\bv(t)\|_{\wi\L^{\frac{q+1}{q}}}^{\frac{q+1}{q}}\d t
	\nonumber\\&:=C{(I_3+I_4)}.
	\end{align*}
	Using the uniform estimate \eqref{eqn-uni-bound} (see \eqref{b2} and \eqref{b6}), the H\"older and interpolation inequalities, and the strong convergence \eqref{eqn-strong-conv}, we arrive at
	\begin{align*}
	{	I_3}&=\int_0^T\||\bv_m(t)|^{q-1}(\bv_m(t)-\bv(t))\|_{\wi\L^{\frac{q+1}{q}}}^{\frac{q+1}{q}}\d t
	\nonumber\\&\leq C\int_0^T\|\bv_m(t)\|_{\wi\L^{q+1}}^{\frac{(q-1)(q+1)}{q}}\|\bv_m(t)-\bv(t)\|_{\wi\L^{q+1}}^{\frac{q+1}{q}}\d t 	\nonumber\\&\leq C\int_0^T\|\bv_m(t)\|_{\wi\L^{q+1}}^{\frac{(q-1)(q+1)}{q}}\|\bv_m(t)-\bv(t)\|_{\wi\L^{r+1}}^{\frac{(q-1)(r+1)}{q(r-1)}}\|\bv_m(t)-\bv(t)\|_{\H}^{\frac{2(r-q)}{q(r-1)}}\d t
	\nonumber\\&\leq C\left(\int_0^T\|\bv_m(t)\|_{\wi\L^{q+1}}^{q+1}\d t\right)^{\frac{q-1}{q}}\left(\int_0^T\|\bv_m(t)-\bv(t)\|_{\H}^2\d t\right)^{\frac{r-q}{q(r-1)}}\nonumber\\&\quad\times\left(\int_0^T\left(\|\bv_m(t)\|_{\wi\L^{r+1}}^{r+1}+\|\bv(t)\|_{\wi\L^{r+1}}^{r+1}\right)\d t\right)^{\frac{q-1}{q(r-1)}}
	\nonumber\\&\to 0\ \text{ as }\ m\to\infty. 
	\end{align*}
	Similar arguments as above yield
	\begin{align*}
	{I_4}&=\int_0^T\|\left(|\bv_m(t)|^{q-1}-|\bv(t)|^{q-1}\right)\bv(t)\|_{\wi\L^{\frac{q+1}{q}}}^{\frac{q+1}{q}}\d t\nonumber\\&\leq  	\int_0^T\left\|\int_0^1\frac{d}{d\theta}\left|\theta\bv_m(t)+(1-\theta)\bv(t)\right|^{q-1}\d\theta\right\|_{\wi\L^{\frac{q+1}{q-1}}}^{\frac{q+1}{q}}\|\bv(t)\|_{\wi\L^{q+1}}^{\frac{q+1}{q}}\d t\nonumber\\&\leq (q-1) \int_0^T\left\|\int_0^1|\theta\bv_m(t)+(1-\theta)\bv(t)|^{q-2}|\bv_m(t)-\bv(t)|\d\theta\right\|_{\wi\L^{\frac{q+1}{q-1}}}^{\frac{q+1}{q}}\|\bv(t)\|_{\wi\L^{q+1}}^{\frac{q+1}{q}}\d t\nonumber\\&\leq C\int_0^T\left(\|\bv_m(t)\|_{\wi\L^{q+1}}+\|\bv(t)\|_{\wi\L^{q+1}}\right)^{\frac{(q-2)(q+1)}{q}}\|\bv_m(t)-\bv(t)\|_{\wi\L^{q+1}}^{\frac{q+1}{q}}\|\bv(t)\|_{\wi\L^{q+1}}^{\frac{q+1}{q}}\d t
		\nonumber\\&\to 0\ \text{ as }\ m\to\infty. 
	\end{align*}
	Therefore, the convergence \eqref{nonlinear-storng} holds and since $\mathcal{P}:\wi\L^{\frac{q+1}{q}}\to  \wi\L^{\frac{q+1}{q}}$ is a bounded linear operator, we immediately have \eqref{eqn-nonlinear-storng}. By the uniqueness of weak limits, we also have $\widetilde{\boldsymbol{\xi}}=\widetilde{\mathcal{C}}(\bv)$. Since the strong convergence implies the norm convergence, from \eqref{nonlinear-storng}, we deduce 
	\begin{align}\label{norm-strong}
	&\int_0^T\|\bv_m(t)\|_{\L^{q+1}}^{q+1}\d t=	\||\bv_m(t)|^{q-1}\bv_m(t)\|_{\mathrm{L}^{\frac{q+1}{q}}(0,T;\wi{\L}^{\frac{q+1}{q}})}^{\frac{q+1}{q}}\nonumber\\& \qquad \to\| |\bv(t)|^{q-1}\bv(t) \|_{\mathrm{L}^{\frac{q+1}{q}}(0,T;\wi{\L}^{\frac{q+1}{q}})}^{\frac{q+1}{q}}=\int_0^T\|\bv(t)\|_{\wi{\L}^{q+1}}^{q+1}\d t \ \text{ as }\ m\to\infty. 
	\end{align}
	
	\emph{Part (v). Weak solution and energy inequality:} 	
Passing limit $m\to\infty$ in \eqref{a6} assures that $\{\bv_m\}$ converges to a weak solution $\bv$ of the periodic problem \eqref{222}, that is, \eqref{3.13} is satisfied.  Let us now find the regularity of time derivative $\partial_t\bv$. For $r\in[3,\infty),$ we consider $\boldsymbol{\phi}\in\mathrm{L}^2(0,T;\V)\cap\mathrm{L}^{r+1}(0,T;\wi\L^{r+1})$ and using a calculation similar to \eqref{eqn-time-der} (except for the trilinear term, see \eqref{be-est}) yields 
\begin{align*}
	|\langle\partial_t\bv,\boldsymbol{\phi}\rangle | 	&\leq\|\f\|_{\V^{\prime}}\|\boldsymbol{\phi}\|_{\V}+\mu\|\nabla\bv\|_{\H}\|\nabla\boldsymbol{\phi}^1\|_{\H}+\|\bv\|_{\widetilde{\L}^{r+1}}\|\bv\|_{\widetilde{\L}^{\frac{2(r+1)}{r-1}}}\|\boldsymbol{\phi}\|_{\V}+\alpha\|\bv\|_{\H}\|\boldsymbol{\phi}\|_{\H}\nonumber\\&\quad+\beta\|\bv\|_{\wi\L^{r+1}}^{r}\|\boldsymbol{\phi}\|_{\wi\L^{r+1}}+|\gamma|\|\bv\|_{\wi\L^{q+1}}^q\|\boldsymbol{\phi}\|_{\wi\L^{q+1}}
	\nonumber\\&\leq\left(\|\f\|_{\V^{\prime}}+\mu\|\bv\|_{\V}+\frac{\alpha}{\sqrt{\lambda_1}}\|\bv\|_{\H}+\|\bv\|_{\widetilde{\L}^{r+1}}^{\frac{r+1}{r-1}}\|\bv\|_{\H}^{\frac{r-3}{r-1}}\right)\|\boldsymbol{\phi}\|_{\V}\nonumber\\&\qquad+\left(\beta\|\bv\|_{\wi\L^{r+1}}^{r}+|\gamma||\Omega|^{\frac{r-q}{r+1}}\|\bv\|_{\wi\L^{r+1}}^{q}\right)\|\boldsymbol{\phi}\|_{\wi\L^{r+1}}. 
\end{align*}
Therefore, we deduce 
\begin{align*}
&\int_0^T	|\langle\partial_t\bv(t),\boldsymbol{\phi}(t)\rangle | \d t
\nonumber\\&\leq \left(\|\f\|_{\mathrm{L}^2(0,T;\V^{\prime})}+\mu\|\bv\|_{\mathrm{L}^2(0,T;\V)}+\frac{\alpha}{\sqrt{\lambda_1}}\|\bv\|_{\mathrm{L}^{\infty}(0,T;\H)}+T^{\frac{r-3}{2(r-1)}}\|\bv\|_{\mathrm{L}^{\infty}(0,T;\H)}^{\frac{r-3}{r-1}}\|\bv\|_{\mathrm{L}^{r+1}(0,T;\wi\L^{r+1})}^{\frac{r+1}{r-1}}\right)\nonumber\\&\quad\times\|\boldsymbol{\phi}\|_{\mathrm{L}^2(0,T;\V)}+\left(\beta\|\bv\|_{\mathrm{L}^{r+1}(0,T;\wi\L^{r+1})}^r+|\gamma||\Omega|^{\frac{r-q}{r+1}}T^{\frac{r-q}{r+1}}\|\bv\|_{\mathrm{L}^{r+1}(0,T;\wi\L^{r+1})}^q\right) \|\boldsymbol{\phi}\|_{\mathrm{L}^{r+1}(0,T;\wi\L^{r+1})},
\end{align*}
for all $\boldsymbol{\phi}\in\mathrm{L}^2(0,T;\V)\cap\mathrm{L}^{r+1}(0,T;\wi\L^{r+1})$. Hence, we get  $\partial_t\bv\in \mathrm{L}^2(0,T;\V^{\prime})+\mathrm{L}^{\frac{r+1}{r}}(0,T;\wi\L^{\frac{r+1}{r}})$. For $r\in[1,3)$, we consider $\boldsymbol{\phi}\in\mathrm{L}^{\frac{4}{4-d}}(0,T;\V)\cap\mathrm{L}^{r+1}(0,T;\wi\L^{r+1}),$  and using a calculation similar to \eqref{eqn-time-der} provides 
\begin{align*}
	|\langle\partial_t\bv,\boldsymbol{\phi}\rangle | 	&\leq\left(\|\f\|_{\V^{\prime}}+\mu\|\bv\|_{\V}+\frac{\alpha}{\sqrt{\lambda_1}}\|\bv\|_{\H}+2^{\frac{d}{2}}\|\bv\|_{\H}^{\frac{4-d}{2}}\|\bv\|_{\V}^{\frac{d}{2}}\right)\|\boldsymbol{\phi}\|_{\V}\nonumber\\&\qquad+\left(\beta\|\bv\|_{\wi\L^{r+1}}^{r}+|\gamma||\Omega|^{\frac{r-q}{r+1}}\|\bv\|_{\wi\L^{r+1}}^{q}\right)\|\boldsymbol{\phi}\|_{\wi\L^{r+1}}. 
\end{align*}
Thus, it is immediate that 
\begin{align*}
	&\int_0^T	|\langle\partial_t\bv(t),\boldsymbol{\phi}(t)\rangle | \d t
	\nonumber\\&\leq \left(\|\f\|_{\mathrm{L}^2(0,T;\V^{\prime})}+\mu\|\bv\|_{\mathrm{L}^2(0,T;\V)}+\frac{\alpha}{\sqrt{\lambda_1}}\|\bv\|_{\mathrm{L}^{\infty}(0,T;\H)}\right)\|\boldsymbol{\phi}\|_{\mathrm{L}^2(0,T;\V)}\nonumber\\&\quad+2^{\frac{d}{2}}\|\bv\|^{\frac{4-d}{2}}_{\mathrm{L}^{\infty}(0,T;\H)}\|\bv\|_{\mathrm{L}^2(0,T;\V)}^{\frac{d}{2}}\|\boldsymbol{\phi}\|_{\mathrm{L}^{\frac{4}{4-d}}(0,T;\V)}\nonumber\\&\quad+\left(\beta\|\bv\|_{\mathrm{L}^{r+1}(0,T;\wi\L^{r+1})}^r+|\gamma||\Omega|^{\frac{r-q}{r+1}}T^{\frac{r-q}{r+1}}\|\bv\|_{\mathrm{L}^{r+1}(0,T;\wi\L^{r+1})}^q\right) \|\boldsymbol{\phi}\|_{\mathrm{L}^{r+1}(0,T;\wi\L^{r+1})},
\end{align*}
for all $\boldsymbol{\phi}\in\mathrm{L}^{\frac{4}{4-d}}(0,T;\V)\cap\mathrm{L}^{r+1}(0,T;\wi\L^{r+1})$. Hence, we get  $\partial_t\bv\in \mathrm{L}^{\frac{4}{d}}(0,T;\V^{\prime})+\mathrm{L}^{\frac{r+1}{r}}(0,T;\wi\L^{\frac{r+1}{r}})$.

 The regularity  $\bv\in\mathrm{L}^{\infty}(0,T;\H)\cap\mathrm{L}^2(0,T;\V)\cap\mathrm{L}^{r+1}(0,T;\widetilde\L^{r+1}))$ and  $\partial_t\bv\in\mathrm{L}^{p}(0,T;\mathbb{V}')+\mathrm{L}^{\frac{r+1}{r}}(0,T;\widetilde\L^{\frac{r+1}{r}})\hookrightarrow \mathrm{L}^{\min\left\{p,\frac{r+1}{r}\right\}}(0,T;\mathbb{V}'+\widetilde\L^{\frac{r+1}{r}}) ,$ where $p$ is defined in \eqref{value-p} imply $\bv\in\W^{1,\min\left\{p,\frac{r+1}{r}\right\}}(0,T;\mathbb{V}'+\widetilde\L^{\frac{r+1}{r}})\hookrightarrow\C([0,T];\mathbb{V}'+\widetilde\L^{\frac{r+1}{r}}) $ (\cite[Theorem 2, pp. 302]{LCE}). Since $\H$ is reflexive and the embedding $\H\hookrightarrow\mathbb{V}'+\widetilde\L^{\frac{r+1}{r}}$ is continuous,   therefore, by an application of \cite[Lemma 8.1, Chapter 3, pp. 275]{JLL+EM-I-72} (also see \cite[Proposition 1.7.1, Chapter 1, pp. 61]{PCAM}) yields $\bv\in\C_w([0,T];\H)$. Therefore, the periodic condition is satisfied in the sense that
 \begin{align*} 
 	\mbox{
 	$(\bv(T),\boldsymbol{\varphi})=(\bv(0),\boldsymbol{\varphi})$ for all $\boldsymbol{\varphi}\in\H$, that is, $(\bv(T)-\bv(0),\boldsymbol{\varphi})=0$ for all $\boldsymbol{\varphi}\in\H.$ }
 	\end{align*}
By 	taking $\boldsymbol{\varphi}= \bv(T)-\bv(0)$, we obtain 
 \begin{align*}
 \mbox{$\bv(T)=\bv(0)$ in $\H$. }
\end{align*}

Using the weakly lowersemicontinuity property of norms and the uniform bounds \eqref{b1} and \eqref{b6}, we further have 
\begin{align*}
	&\|\bv(t)\|_{\H}^2+\mu\int_0^t\|\bv(s)\|_{\V}^2\d s+\beta\int_0^t\|\bv(s)\|_{\wi\L^{r+1}}^{r+1}\d s\nonumber\\&\leq \liminf_{m\to\infty}\left[
	\|\bv_m(t)\|_{\H}^2+\mu\int_0^t\|\bv_m(s)\|_{\V}^2\d s+\beta\int_0^t\|\bv_m(s)\|_{\wi\L^{r+1}}^{r+1}\d s\right]\leq \left(\frac{1}{\mu\lambda_1T}+2\right)
	\mathcal{K},
\end{align*}
	for all $t\in[0,T]$. We infer from \eqref{a2} that 
	\begin{align}\label{eqn-finite-ener}
		&\|\bv_m(t)\|_{\H}^2+2\mu\int_0^t\|\bv_m(s)\|_{\V}^2\d s+2\alpha\int_0^t\|\bv_m(s)\|_{\H}^2\d s+2\beta\int_0^t\|\bv_m(s)\|_{\wi\L^{r+1}}^{r+1}\d s\nonumber\\&\quad+2\gamma\int_0^t\|\bv_m(s)\|_{\wi\L^{q+1}}^{q+1}\d s\nonumber\\&=\|\bv_m(0)\|_{\H}^2+2\int_0^t\langle\f(s),\bv_m(s)\rangle\d s,
	\end{align}
for all $t\in[0,T]$.	The existence of \emph{Leray-Hopf weak solution}  holds true for all $r\in[1,\infty)$ and $d\in\{2,3\}$. Taking $t=T$,  using the fact that $\bv_m(0)=\bv_m(T)$ and then taking $\liminf$ on both sides of \eqref{eqn-finite-ener}, we find 
\begin{align}\label{eqn-ener-in}
&2\mu\int_0^T\|\bv(t)\|_{\V}^2\d t+2\alpha\int_0^T\|\bv(t)\|_{\H}^2\d t+2\beta\int_0^T\|\bv(t)\|_{\wi\L^{r+1}}^{r+1}\d t\nonumber\\	&\leq \liminf_{m\to\infty}\bigg\{2\mu\int_0^T\|\bv_m(t)\|_{\V}^2\d t+2\alpha\int_0^T\|\bv_m(t)\|_{\H}^2\d t+2\beta\int_0^T\|\bv_m(t)\|_{\wi\L^{r+1}}^{r+1}\d t\bigg\}\nonumber\\&\nonumber\\&= \liminf_{m\to\infty}\bigg\{2\int_0^T\langle\f(t),\bv_m(t)\rangle\d t-2\gamma\int_0^t\|\bv_m(t)\|_{\wi\L^{q+1}}^{q+1}\d t\bigg\}
\nonumber\\&=2\int_0^T\langle\f(t),\bv(t)\rangle\d t-2\gamma\int_0^T\|\bv(t)\|_{\wi\L^{q+1}}^{q+1}\d t,
\end{align}
	where we have used the strong convergence  \eqref{norm-strong} also. Therefore, the energy inequality \eqref{ener-inequality}  follows from \eqref{eqn-ener-in}. Moreover the estimate \eqref{p36} is also satisfied.

		\vskip 0.1 cm
	\noindent
		\textbf{Step (2):} \emph{Energy equality:} 
	Let us now consider the case $r\in[1,\infty)$ when $d=2$ and $r\in[3,\infty)$ when $d=3$. We show that $\v(\cdot)$  satisfies the energy equality. It should be emphasized that this equality does not follow immediately. To derive it, one can adopt the approximation procedure introduced in \cite{CLF}. In that work, the authors constructed a sequence of approximations to  $\bv(\cdot)$ on bounded domains, where the approximating functions are uniformly bounded and converge simultaneously in both Sobolev and Lebesgue spaces.	Using the techniques  in \cite{CLF,galdi,MTM7}, we infer that the solution $\bv\in\C([0,T];\H)$ and satisfies the energy equality  	\begin{align*}
			&\|\bv(t)\|^2_{\H}+2\mu \int_0^t\|\bv(s)\|^2_{\V}\d s+2\alpha\int_0^t\|\bv(s)\|^2_{\H}\d s+2\beta\int_0^t\|\bv(s)\|_{\widetilde{\L}^{r+1}}^{r+1}\d s+2\gamma\int_0^t\|\bv(s)\|_{\widetilde{\L}^{q+1}}^{q+1}\d s\nonumber\\&=\|\bv(0)\|_{\H}^2+2\int_0^t\langle\f(s),\bv(s)\rangle\d s,
		\end{align*}
		for all $t\in[0,T]$. For $r\in[3,\infty),$ owing to the regularity  $\bv\in\mathrm{L}^{\infty}(0,T;\H)\cap\mathrm{L}^2(0,T;\V)\cap\mathrm{L}^{r+1}(0,T;\widetilde\L^{r+1}))$ and  $\partial_t\bv\in\mathrm{L}^{2}(0,T;\mathbb{V}')+\mathrm{L}^{\frac{r+1}{r}}(0,T;\widetilde\L^{\frac{r+1}{r}}),$ we obtain the energy equality by an application of Theorem \ref{Thm-Abs-cont}. The case of $r\in[1,3)$ when $d=2$ is immediate (\cite[Theorem 4.1]{galdi}) since $\mathrm{L}^{\infty}(0,T;\H)\cap\mathrm{L}^2(0,T;\V)\hookrightarrow\mathrm{L}^4(0,T;\wi\L^4)$ by an application of the Ladyzhenskaya inequality.

		\vskip 0.1 cm
		\noindent\textbf{Step (3):} \emph{Uniqueness.} 
		Let us now show the uniqueness of weak solutions. We assume that $\bv_1(\cdot)$ and $\bv_2(\cdot)$ are two weak solutions of problem \eqref{222}. Then, $\bv(\cdot)=\bv_1(\cdot)-\bv_2(\cdot)$ satisfies: 
		\begin{equation}\label{a15}
			\left\{
			\begin{aligned}
				\partial_t\bv(t)+\mu\A\bv(t)+(\B(\bv_1(t))-\B(\bv_2(t)))&+\alpha\bv(t)+\beta(\mathcal{C}(\bv_1(t))-\mathcal{C}(\bv_2(t)))\\+\gamma(\widetilde{\mathcal{C}}(\bv_1(t))-\widetilde{\mathcal{C}}(\bv_2(t)))&=\mathbf{0},  \ \text{ in }\ \V^{\prime}+\wi\L^{\frac{r+1}{r}},\\
				\bv(0)&=\bv(T),
			\end{aligned}
			\right.
		\end{equation}
		for a.e. $t\in[0,T]$, where $\bv_1,\bv_2\in\mathrm{L}^{\infty}(0,T;\H)\cap\mathrm{L}^2(0,T;\V)\cap\mathrm{L}^{r+1}(0,T;\widetilde\L^{r+1})),$  with $\partial_t\bv_1,\partial_t\bv_2\in\mathrm{L}^{2}(0,T;\mathbb{V}')+\mathrm{L}^{\frac{r+1}{r}}(0,T;\widetilde\L^{\frac{r+1}{r}}).$ Using the absolutely continuity of the mapping $[0,T]\ni t\mapsto\|\bv(t)\|_{\H}^2\in\mathbb{R}$ and taking the inner product with $\bv(\cdot)$ to the first equation in \eqref{a15}, we find 
		\begin{align}\label{a16}
			&\frac{1}{2}\frac{\d}{\d t}\|\bv(t)\|_{\H}^2+\mu\|\bv(t)\|_{\V}^2+\alpha\|\bv(t)\|_{\H}^2+\beta\langle \mathcal{C}(\bv_1(t))-\mathcal{C}(\bv_2(t)),\bv(t)\rangle \nonumber\\&= -\gamma\langle \widetilde{\mathcal{C}}(\bv_1(t))-\widetilde{\mathcal{C}}(\bv_2(t)),\bv(t)\rangle-\langle\B(\bv(t),\bv_2(t)),\bv(t)\rangle,
		\end{align}
	for a.e. $t\in[0,T]$, where we have used the fact that $\langle\B(\bv_1,\bv),\bv\rangle=0$. The estimate \eqref{2.23} yields 
	\begin{align}\label{eqn-difference-0}
	\beta\langle \mathcal{C}(\bv_1)-\mathcal{C}(\bv_2),\bv\rangle 
\leq \frac{\beta}{2}\||\bv_1|^{\frac{r-1}{2}}\bv\|_{\H}^2+\frac{\beta}{2}\||\bv_2|^{\frac{r-1}{2}}\bv\|_{\H}^2. 
	\end{align}
Let us fix $\boldsymbol{\psi}(\u)=|\u|^{q-1}\u$. Then, we consider 
\begin{align}\label{eqn-difference-1}
 -&\gamma\langle \widetilde{\mathcal{C}}(\bv_1)-\widetilde{\mathcal{C}}(\bv_2),\bv\rangle\nonumber\\&= -\gamma	\langle|\bv_1|^{q-1}\bv_1-|\bv_2|^{q-1}\bv_2,\bv\rangle=\left\langle\int_0^1\frac{\d}{\d\theta}\boldsymbol{\psi}(\theta\bv_1+(1-\theta)\bv_2)\d\theta,\bv\right\rangle \nonumber\\&=-\gamma\left\langle\int_0^1\nabla\boldsymbol{\psi}(\theta\bv_1+(1-\theta)\bv_2)\cdot\bv\d\theta, \bv\right\rangle\nonumber\\&=-\gamma\left\langle\int_0^1|\theta\bv_1+(1-\theta)\bv_2|^{q-1}\bv\d\theta,\bv\right\rangle\nonumber\\&\quad+(q-1)\left\langle\int_0^1|\theta\bv_1+(1-\theta)\bv_2|^{q-3}\left((\theta\bv_1+(1-\theta)\bv_2)\cdot\bv\right)(\theta\bv_1+(1-\theta)\bv_2)\d\theta,\bv\right\rangle\nonumber\\&\leq|\gamma| q\left\langle\left(|\bv_1|+|\bv_2|\right)^{q-1}|\bv|,|\bv|\right\rangle
\leq |\gamma|q2^{q-2}\left\langle\left(|\bv_1|^{q-1}+|\bv_2|^{q-1}\right)|\bv|,|\bv|\right\rangle.
\end{align}
 By  using  H\"older's and Young's inequalities, we estimate  $ |\gamma|q2^{q-2}\langle|\bv_1|^{q-1}|\bv|,|\bv|\rangle$ as 
 \begin{align*}
 	|\gamma|q2^{q-2}\left\langle|\bv_1|^{q-1}|\bv|,|\bv|\right\rangle&=  	|\gamma|q2^{q-2}\int_{\Omega} |\bv_1(x)|^{q-1}|\bv(x)|^{\frac{2(q-1)}{(r-1)}}|\bv(x)|^{\frac{2(r-q)}{(r-1)}}\d x\nonumber\\&\leq 	|\gamma|q2^{q-2}\left(\int_{\Omega}|\bv_1(x)|^{r-1}|\bv(x)|^2\d x\right)^{\frac{q-1}{r-1}}\left(\int_{\Omega}|\bv(x)|^2\d x\right)^{\frac{r-q}{r-1}}
 	\nonumber\\&\leq\frac{\beta}{2}\||\bv_1|^{\frac{r-1}{2}}\bv\|_{\H}^2+\left(|\gamma|q2^{q-2}\right)^{\frac{r-1}{r-q}}\left(\frac{r-q}{r-1}\right)\left(\frac{2(q-1)}{\beta(r-1)}\right)^{\frac{q-1}{r-1}}\|\bv\|_{\H}^2. 
 \end{align*}
 A similar calculation yields 
 \begin{align*}
 	|\gamma|q2^{q-2}\left\langle|\bv_1|^{q-1}|\bv|,|\bv|\right\rangle&\leq \frac{\beta}{4}\||\bv_2|^{\frac{r-1}{2}}\bv\|_{\H}^2+\left(|\gamma|q2^{q-2}\right)^{\frac{r-1}{r-q}}\left(\frac{r-q}{r-1}\right)\left(\frac{4(q-1)}{\beta(r-1)}\right)^{\frac{q-1}{r-1}}\|\bv\|_{\H}^2. 
 \end{align*}
 Therefore, we deduce from \eqref{eqn-difference-1} that 
 \begin{align}\label{eqn-difference-2}
 	 -\gamma\langle \widetilde{\mathcal{C}}(\bv_1)-\widetilde{\mathcal{C}}(\bv_2),\bv\rangle&
 	\leq\frac{\beta}{2}\||\bv_1|^{\frac{r-1}{2}}\bv\|_{\H}^2+\frac{\beta}{4}\||\bv_2|^{\frac{r-1}{2}}\bv\|_{\H}^2\nonumber\\&\quad +\left(|\gamma|q2^{q-2}\right)^{\frac{r-1}{r-q}}\left(\frac{r-q}{r-1}\right)\left(\frac{2(q-1)}{\beta(r-1)}\right)^{\frac{q-1}{r-1}}\left[1+2^{\frac{q-1}{r-1}}\right]\|\bv\|_{\H}^2. 
 \end{align}
For $r>3$, we estimate the trilinear form $-\langle\B(\bv,\bv_2),\bv\rangle$ by using H\"older's and Young's inequalities as 
\begin{align*}
	-\langle\B(\bv,\bv_2),\bv\rangle&=\langle\B(\bv,\bv),\bv_2\rangle\leq\|\bv\|_{\V}\||\bv_2|\bv\|_{\H}\leq\frac{\mu}{2}\|\bv\|_{\V}^2+\frac{1}{2\mu}\|\bv_2\bv\|_{\H}^2. 
\end{align*}
Once again using H\"older's and Young's inequalities, we calculate  $\frac{1}{2\mu}\|\bv_2\bv\|_{\H}^2$ as 
\begin{align*}
	\frac{1}{2\mu}\||\bv_2|\bv\|_{\H}^2&=\frac{1}{2\mu}\int_{\Omega}|\bv_2(x)|^2|\bv(x)|^{\frac{4}{r-1}}|\bv(x)|^{\frac{2(r-3)}{r-1}}\d x\nonumber\\&\leq\frac{1}{2\mu}\left(\int_{\Omega}|\bv_2(x)|^{r-1}|\bv(x)|^2\d x\right)^{\frac{2}{r-1}}\left(\int_{\Omega}|\bv(x)|^2\d x\right)^{\frac{2(r-3)}{r-1}}\nonumber\\&\leq\frac{\beta}{4}\||\bv_2|^{\frac{r-1}{2}}\bv\|_{\H}^2+\left(\frac{1}{2\mu}\right)^{\frac{r-1}{r-3}}\left(\frac{r-3}{r-1}\right)\left(\frac{8}{\beta(r-1)}\right)^{\frac{2}{r-3}}\|\bv\|_{\H}^2.
\end{align*}
Therefore, we have the estimate 
\begin{align}\label{eqn-difference-3}
	-\langle\B(\bv,\bv_2),\bv\rangle&\leq  \frac{\mu}{2}\|\bv\|_{\V}^2+\frac{\beta}{4}\||\bv_2|^{\frac{r-1}{2}}\bv\|_{\H}^2+\left(\frac{1}{2\mu}\right)^{\frac{r-1}{r-3}}\left(\frac{r-3}{r-1}\right)\left(\frac{8}{\beta(r-1)}\right)^{\frac{2}{r-3}}\|\bv\|_{\H}^2.
\end{align}
Using \eqref{eqn-difference-0}, \eqref{eqn-difference-2}, \eqref{eqn-difference-3} in \eqref{a16}, we deduce 
\begin{align}\label{eqn-difference-4}
	&\frac{1}{2}\frac{\d}{\d t}\|\bv(t)\|_{\H}^2+\frac{\mu}{2}\|\bv(t)\|_{\V}^2+\alpha\|\bv(t)\|_{\H}^2
	\nonumber\\&\leq \left\{\left(\frac{1}{2\mu}\right)^{\frac{r-1}{r-3}}\left(\frac{r-3}{r-1}\right)\left(\frac{8}{\beta(r-1)}\right)^{\frac{2}{r-3}}\right.\nonumber\\&\qquad\left.+ \left(|\gamma|q2^{q-2}\right)^{\frac{r-1}{r-q}}\left(\frac{r-q}{r-1}\right)\left(\frac{2(q-1)}{\beta(r-1)}\right)^{\frac{q-1}{r-1}}\left[1+2^{\frac{q-1}{r-1}}\right]\right\}\|\bv(t)\|_{\H}^2,
\end{align}
for a.e. $t\in[0,T]$. From the above expression, it follows immediately, by the Poincar\'e inequality, that
\begin{align}\label{eqn-diff-5}
	\frac{\d}{\d t}\|\bv(t)\|_{\H}^2+\left(\mu\lambda_1+2\alpha-\zeta\right) \|\bv(t)\|_{\H}^2\leq 0
\end{align}
where $\zeta= 2\left\{\left(\frac{1}{2\mu}\right)^{\frac{r-1}{r-3}}\left(\frac{r-3}{r-1}\right)\left(\frac{8}{\beta(r-1)}\right)^{\frac{2}{r-3}}+ \left(|\gamma|q2^{q-2}\right)^{\frac{r-1}{r-q}}\left(\frac{r-q}{r-1}\right)\left(\frac{2(q-1)}{\beta(r-1)}\right)^{\frac{q-1}{r-1}}\left[1+2^{\frac{q-1}{r-1}}\right]\right\}$.
	
		For $\mu\lambda_1+2\alpha>\zeta$, from \eqref{eqn-diff-5}, we arrive at 
		\begin{align}\label{a19}
			\|\bv(t)\|_{\H}^2\leq\|\bv(0)\|_{\H}^2e^{-Lt}, \ \text{ where }\ L=\left(\mu\lambda_1+2\alpha-\zeta\right)>0.
		\end{align}
	for all $t\in[0,T]$.	Taking $t=T$ in \eqref{a19} and using the fact that $\bv(T)=\bv(0)$, we obtain $\|\bv(0)\|_{\H}=0$. Using this fact in \eqref{a19}, we obtain $\bv_1(t,x)=\bv_2(t,x)$, for $t\in[0,T]$, for a.e. $x\in\Omega$.  
		
		For $r>3$, one can estimate $\langle\B(\bv,\bv),\bv_2\rangle$ in the following way also: 
		\begin{align*}
			\langle\B(\bv,\bv),\bv_2\rangle&\leq\|\bv\|_{\V}\||\bv_2|\bv\|_{\H}\leq\frac{1}{\beta}\|\bv\|_{\V}^2+\frac{\beta}{4}\|\bv_2\bv\|_{\H}^2\nonumber\\&=\frac{1}{\beta}\|\bv\|_{\V}^2+\frac{\beta}{4}\int_{\Omega}|\bv(x)|^2|\left(|\bv_2(x)|^{r-1}+1\right)\frac{|\bv_2(x)|^2}{|\bv_2(x)|^{r-1}+1}\d x\nonumber\\&\leq \frac{1}{\beta}\|\bv\|_{\V}^2+\frac{\beta}{4}\int_{\Omega}|\bv(x)|^2|\bv_2(x)|^{r-1}\d x+ \frac{\beta}{4}\int_{\Omega}|\bv(x)|^2\d x, 
		\end{align*}
		where we have used the fact that $\left\|\frac{|\v|^2}{|\v|^{r-1}+1}\right\|_{\widetilde{\L}^{\infty}}<1$, for $r\geq 3$. Using this estimate instead of \eqref{eqn-difference-3}, for $\beta\mu>1$, we derive from \eqref{eqn-difference-4} for all $t\in[0,T]$ that 
			\begin{align*}
			\|\bv(t)\|_{\H}^2\leq\|\bv(0)\|_{\H}^2e^{-L_1t}, \ \mbox{ where $L_1=2\left(\left(\mu-\frac{1}{\beta}\right)\lambda_1+\alpha-\eta\right)>0,$ 	}
		\end{align*}
	 and $\eta= \left\{\frac{\beta}{4}+ \left(|\gamma|q2^{q-2}\right)^{\frac{r-1}{r-q}}\left(\frac{r-q}{r-1}\right)\left(\frac{2(q-1)}{\beta(r-1)}\right)^{\frac{q-1}{r-1}}\left[1+2^{\frac{q-1}{r-1}}\right]\right\}$. The uniqueness follows in this case also by taking $\left(\mu-\frac{1}{\beta}\right)\lambda_1+\alpha>\eta$ and $\beta\mu>1$.   
		
		For $r=3$ and $\beta\mu>1$, using \eqref{2.23} and \eqref{eqn-difference-2}, we have 
		\begin{align*}
			&\frac{1}{2}\frac{\d}{\d t}\|\bv(t)\|_{\H}^2+\mu\|\bv(t)\|_{\V}^2+\alpha\|\bv(t)\|_{\H}^2+\frac{\beta}{2}\||\bv_1(t)|\bv(t)\|_{\H}^2+\frac{\beta}{2}\||\bv_2(t)|\bv(t)\|_{\H}^2\nonumber\\&\leq-\gamma\langle \widetilde{\mathcal{C}}(\bv_1(t))-\widetilde{\mathcal{C}}(\bv_2(t)),\bv(t)\rangle  +\langle\B(\bv(t),\bv(t)),\bv_2(t)\rangle\nonumber\\&\leq \frac{\beta}{2}\||\bv_1(t)|\bv(t)\|_{\H}^2+\frac{\beta}{4}\||\bv_2(t)|\bv(t)\|_{\H}^2+\kappa\|\bv(t)\|_{\H}^2+ \|\bv(t)\|_{\V}\||\bv_2(t)|\bv(t)\|_{\H}\nonumber\\&\leq \frac{\beta}{2}\||\bv_1(t)|\bv(t)\|_{\H}^2+\frac{\beta}{2}\||\bv_1(t)|\bv(t)\|_{\H}^2+\kappa\|\bv(t)\|_{\H}^2+\frac{1}{\beta}\|\bv(t)\|_{\V}^2,
		\end{align*}
		for a.e. $t\in[0,T]$, where $\kappa=\left(|\gamma|q2^{q-2}\right)^{\frac{2}{3-q}}\left(\frac{3-q}{2}\right)\left(\frac{q-1}{\beta}\right)^{\frac{q-1}{2}}\left[1+2^{\frac{q-1}{2}}\right]$. Thus, for $\beta\mu>1$ and $\alpha+\left(\mu-\frac{1}{\beta}\right)\lambda_1>\kappa$, we have 
		\begin{align*}
			\|\bv(t)\|_{\H}^2\leq\|\bv(0)\|_{\H}^2e^{-2\wi L t}, \ \text{ for all } \ t\in[0,T], \ \text{ where }\ \wi L=\alpha+\left(\mu-\frac{1}{\beta}\right)\lambda_1-\kappa,
		\end{align*}
		and the uniqueness of weak solutions follows. 
	\end{proof} 
	
\begin{remark}
	For $\gamma=0$ in \eqref{222}, that is, for CBF equations with $r>3$, uniqueness of weak solutions is ensured under the condition $\mu\lambda_1+2\alpha>2\left(\frac{1}{2\mu}\right)^{\frac{r-1}{r-3}}\left(\frac{r-3}{r-1}\right)\left(\frac{4}{\beta(r-1)}\right)^{\frac{2}{r-3}}$, and for $r=3$, uniqueness holds provided that $2\beta\mu\geq 1$. 
\end{remark}

{
\begin{appendix}
		\renewcommand{\thesection}{\Alph{section}}
		\numberwithin{equation}{section}

		\section{Uniqueness for $d=2,3$ and $r\in[1,3]$}\label{App} 
		
		For $d=2,3$ and $r\in[1,3]$ (including the case $r=3$, which is not covered in Theorem \ref{thm2.11})), the difficulty of the uniqueness problem is comparable to that of the 2D and 3D time-periodic NSE. Under a suitable smallness assumption on the data $\f\in\mathrm{L}^2(0,T;\H)$, by  employing a fixed point argument in appropriate function spaces, we obtain the uniqueness of strong solutions. For $\f\in\mathrm{L}^2(0,T;\H)$, we say that $\bv\in\mathrm{C}([0,T];\V)\cap\mathrm{L}^2(0,T;\D(\A)),$ with $\partial_t\bv\in\mathrm{L}^2(0,T;\H)$ is a \emph{time-periodic strong solution}  to the system (\ref{222}) if $\bv(0)=\bv(T)$  and  (\ref{222}) is satisfied  in $\H$ for a.e. $t\in(0,T)$.

		
			\begin{theorem}\label{thm-app-2}
			For $d=2,3$ and $r\in[1,3]$, let  $\f\in\mathrm{L}^2(0,T;\H)$  be given. If $\|\f\|_{\mathrm{L}^2(0,T;\H)}$ is sufficiently small, then  there exists a unique  time-periodic strong solution $\bv\in \mathrm{C}([0,T];\V)\cap \mathrm{L}^2(0,T;\mathrm{D}(\A)) \cap \mathrm{H}^1(0,T;\mathbb{H})$ for the problem \eqref{222}. Moreover, the solution satisfies the following estimate:
			\begin{align}\label{ener-storng}
				\|\bv\|_{\mathrm{L}^{\infty}(0,T;\V)}+\|\bv\|_{\mathrm{L}^2(0,T;\D(\A))}+\|\partial_t\bv\|_{\mathrm{L}^2(0,T;\H)}\leq C\|\f\|_{\mathrm{L}^2(0,T;\H)}. 
			\end{align}
		\end{theorem}
		\begin{proof}
		\textbf{Step 1: The linear problem.} We first consider the following linear time-periodic problem in $\mathbb{H}$: 
		\begin{equation}\label{linear-periodic}
			\left\{
			\begin{aligned}
				\partial_t\u(t)+ (\mu \A + \alpha \mathrm{I})\u (t)&= \boldsymbol{h}(t), \ \text{ for a.e.}\ t\in(0,T),\\
				\u(0) &= \u(T),
			\end{aligned}
			\right.
		\end{equation}
		where $\A$ is the Stokes operator and $\boldsymbol{h}\in\mathrm{L}^2(0,T;\H)$. Since $(\lambda_k,\w_k)$, $k\in\mathbb{N}$ is an eigenpair for $\A$, then for the operator $	\mathcal{L} := \mu \A + \alpha \mathrm{I}$, we have  $\mathcal{L} \w_k = \omega_k \w_k,$ where $\omega_k := \mu \lambda_k + \alpha > 0.$
	Using the expansions
		\begin{align}\label{sol-1}
		\u(t) = \sum_{k=1}^\infty u_k(t)\w_k,
		\  \
		\boldsymbol{h}(t) = \sum_{k=1}^\infty h_k(t)\w_k,
		\end{align}
		where $u_k=(\u,\w_k)$ and $h_k=(\boldsymbol{h},\w_k)$ in \eqref{linear-periodic} yields, for each $k$,
		\begin{equation}\label{scalar-ode}
			\dot u_k(t) + \omega_k u_k(t) = h_k(t),
			\ 
			u_k(0) = u_k(T).
		\end{equation}
		The unique solution of \eqref{scalar-ode} is given by 
		\begin{align}\label{linear-solution}
		u_k(t) = e^{-\omega_k t}\left(u_k(0)+	\int_0^t e^{\omega_k s} h_k(s)\d s\right),
		\end{align}
	for all $t\in[0,T]$.	Let us set  $t=T$ and use the periodicity condition $u_k(T)=u_k(0)$ to obtain
	\begin{align*}
		u_k(0)=e^{-\omega_k T}\left(u_k(0)+\int_0^T e^{\omega_k s} h_k(s)\d s\right).
	\end{align*}
		Solving for $u_k(0)$, we find 
	\begin{align*}
		u_k(0)=\frac{e^{-\omega_k T}}{1 - e^{-\omega_k T}}\int_0^T e^{\omega_k s} h_k(s)\d s.
	\end{align*}
		Substituting this expression in \eqref{linear-solution} yields 
	\begin{align}\label{sol-0}
	u_k(t)=\frac{e^{-\omega_k t}}{1-e^{-\omega_k T}}	\left[	\int_0^t e^{\omega_k s} h_k(s)\d s+	e^{-\omega_k T}	\int_t^T e^{\omega_k s} h_k(s)\d s\right].
	\end{align}
		Equivalently, we have 
		\begin{align}\label{sol-2}
		u_k(t)=\int_0^T K_k(t,s) h_k(s)\d s,
		\end{align}
		where the scalar periodic Green kernel is
		\begin{align*}
		K_k(t,s)
		=\left\{\begin{array}{ll}\frac{e^{-\omega_k (t-s)}}{1 - e^{-\omega_k T}},&
			0\leq s \le t,\\
		\frac{e^{-\omega_k (t-s+T)}}{1 - e^{-\omega_k T}},&t< s\leq T.
	\end{array}\right. 
		\end{align*}

		Since $\omega_k>0$, we know that $0<e^{-\omega_k T}<1,$ which implies $1-e^{-\omega_k T}\ge 1-e^{-\omega_1 T}=c>0.$ Therefore, we get $\frac{1}{1-e^{-\omega_k T}}	\le C,$ uniformly in $k$.
		Using the Cauchy-Schwarz inequality, we deduce 
	\begin{align*}
		\left|\int_0^t e^{\omega_k s} h_k(s)\d s \right| \leq \left( \int_0^T e^{2\omega_k s}ds \right)^{1/2}	\|h_k\|_{\mathrm{L}^2(0,T)} \le \frac{e^{\omega_k T}}{\sqrt{2\omega_k}} \|h_k\|_{\mathrm{L}^2(0,T)}.
	\end{align*}
		A similar estimate holds for the second  term in the right hand side of \eqref{sol-0}.
		Therefore, we infer 
	\begin{align*}
		u_k^2(t) \le \frac{C}{\omega_k}	\|h_k\|_{\mathrm{L}^2(0,T)}^2.
	\end{align*}
		Since $\omega_k=\mu\lambda_k+\alpha\ge \mu\lambda_k,$ we have	$\frac{\lambda_k}{\omega_k} \le \frac{1}{\mu}.$
Using this fact, we deduce 
		\begin{align*}
		\lambda_k u_k^2(t) \le C\|h_k\|_{\mathrm{L}^2(0,T)}^2.
	\end{align*}
	Note that 
	$\|\bv(t)\|_{\V}^2=\sum_{k=1}^\infty\lambda_k u_k^2(t).$	Summing over $k$ in the above expression, we find 
	\begin{align*}
	\sup_{t\in[0,T]}	\|\u(t)\|_{\V}^2 \le C	\sum_{k=1}^\infty \|h_k\|_{\mathrm{L}^2(0,T)}^2 = C	\|\boldsymbol{h}\|_{\mathrm{L}^2(0,T;\mathbb{H})}^2.
	\end{align*}

	Multiplying \eqref{scalar-ode} by $\omega_k u_k$, we get  
	\begin{align}\label{ener-est-2}
		\frac{\omega_k}{2}\frac{\d}{\d t}(u_k^2(t)) + \omega_k^2 u_k^2(t) = \omega_k u_k (t)h_k(t),
	\end{align}
	for a.e. $t\in(0,T)$.
		Integrating over time from $0$ to $T$ in \eqref{ener-est-2}, we find 
	\begin{align*}
		\frac{\omega_k}{2}\big(u_k^2(T)-u_k^2(0)\big)+\omega_k^2 \int_0^T u_k^2(t) \d t= \omega_k \int_0^T u_k(t) h_k(t) \d t.
	\end{align*}
		Using periodicity $u_k(T)=u_k(0)$, the boundary term vanishes and by applying Young's inequality, we obtain 
		\begin{align}\label{energy-1}
		\omega_k^2 \int_0^T u_k^2(t) \d t =	\omega_k \int_0^T u_k(t) h_k(t) \d t\leq \frac{\omega_k^2}{2} \int_0^T u_k^2(t) \d t+\frac{1}{2} \int_0^T h_k^2(t) \d t.
		\end{align}
	Since	$\omega_k=\mu\lambda_k+\alpha	\ge \mu\lambda_k,$ we have $	\lambda_k^2\le	\frac{1}{\mu^2}\omega_k^2.$ Therefore, from \eqref{energy-1}, we infer 
	\begin{align*}
		\lambda_k^2 \int_0^T u_k^2(t) \d t	\le	\frac{1}{\mu^2}	\omega_k^2 \int_0^T u_k^2(t) \d t	\le \frac{1}{\mu^2} 	\int_0^T h_k^2(t) \d t.
	\end{align*}
	Note  that  $	\|\A\u\|_{\mathrm{L}^2(0,T;\mathbb{H})}^2= \sum_{k=1}^\infty \lambda_k^2 \int_0^T u_k^2(t) \d t.$	Summing over $k$ in the above expression, we arrive at 
	\begin{align*}
	\|\A\u\|_{\mathrm{L}^2(0,T;\mathbb{H})}^2 \le \frac{1}{\mu^2} \|\boldsymbol{h}\|_{\mathrm{L}^2(0,T;\mathbb{H})}^2,
		\end{align*}
		so that $\u\in \mathrm{L}^2(0,T;\mathrm{D}(\A)).$ From \eqref{linear-periodic}, we have 
	$	\partial_t\u= \boldsymbol{h} - \mathcal{L}\u.$ Since $\boldsymbol{h} \in \mathrm{L}^2(0,T;\mathbb{H})$ and $\mathcal{L}\u \in \mathrm{L}^2(0,T;\mathbb{H})$, we obtain $\partial_t\u\in\mathrm{L}^2(0,T;\mathbb{H})$
		and
		\begin{align*}
		\|\u\|_{\mathrm{L}^2(0,T;\mathrm{D}(\A))} +	\|\partial_t\u\|_{\mathrm{L}^2(0,T;\mathbb{H})} \le C\|\boldsymbol{h}\|_{\mathrm{L}^2(0,T;\mathbb{H})}.
		\end{align*}
		Therefore the linear problem \eqref{linear-periodic} admits a unique time-periodic strong solution (see \eqref{sol-1} and \eqref{sol-2})
	\begin{align*}
	\u \in \mathcal{X}:=\mathrm{C}([0,T];\V)\cap \mathrm{L}^2(0,T;\mathrm{D}(\A)) \cap \mathrm{H}^1(0,T;\mathbb{H}),
\end{align*}
		with estimate
	\begin{align*}
		\|\u\|_{\mathcal{X}} \le C \|\boldsymbol{h}\|_{\mathrm{L}^2(0,T;\mathbb{H})}.
	\end{align*}
	We define $\mathcal{X}_{\mathrm{per}}:=\left\{\u\in\mathcal{X}:\u(0)=\u(T)\right\}$. Let us denote the solution operator $\mathcal{S}:\mathrm{L}^2(0,T;\H)\to\mathcal{X}_{\mathrm{per}}$ by
	\begin{align*}
		\mathcal{S}(\boldsymbol{h})=\u.
	\end{align*}

		\textbf{Step 2: The nonlinear problem.} 
	Let us now consider the nonlinear problem \eqref{222} for $d=3$ and $r\in[1,3]$. The case for $d=3$ and $r\in[1,3]$ is similar and we omit here. We rewrite \eqref{222}  for a.e. $t\in(0,T)$ in $\H$ as
		\begin{align*}
		\partial_t\bv(t)+(\mu\A+\mathrm{I})\bv(t)=\f(t)-\B(\bv(t))-\beta\mathcal{C}(\bv(t))-\gamma\widetilde{\mathcal{C}}(\bv(t)).
	\end{align*}
		We define the mapping
	\begin{align*}
		\Phi(\bv)=	\mathcal{S}\left(	\f-\B(\bv)-\beta\mathcal{C}(\bv)-\gamma\widetilde{\mathcal{C}}(\bv)	\right).
	\end{align*}
		Note that the time-periodic strong solutions of \eqref{222} correspond to  fixed points $\bv=\Phi(\bv).$
We estimate $\|\B(\bv)\|_{\mathrm{L}^2(0,T;\mathbb{H})}$ using H\"older's and Agmon's inequalities as 
	\begin{align*}
	\|\B(\bv)\|_{\mathrm{L}^2(0,T;\mathbb{H})}^2&\leq\int_0^T\|(\bv(t)\cdot\nabla)\bv(t)\|_{\H}^2\d t\leq \int_0^T\|\bv(t)\|_{\wi\L^{\infty}}^2\|\nabla\bv(t)\|_{\H}^2\d t \nonumber\\&\leq C\sqrt{T}\sup_{t\in[0,T]}\|\bv(t)\|_{\V}^3\left(\int_0^T\|\A\bv(t)\|_{\H}^2\d t\right)^{1/2},
	\end{align*}
	so that 
	\begin{align*}
		\|\B(\bv)\|_{\mathrm{L}^2(0,T;\mathbb{H})}\leq C\left(\|\bv\|_{\mathrm{L}^{\infty}(0,T;\V)}^2+\|\bv\|_{\mathrm{L}^2(0,T;\D(\A))}^2\right)\leq C\|\bv\|_{\mathcal{X}}^2. 
	\end{align*}
	An application of Sobolev's inequality yields $\|\bv\|_{\wi\L^{2r}}\leq C\|\bv\|_{\V}$ for $r\in[1,3]$, so that 
	\begin{align*}
		\|\mathcal{C}(\bv)\|_{\mathrm{L}^2(0,T;\H)}\leq C\|\bv\|_{\mathrm{L}^{2r}(0,T;\wi\L^{2r})}^r\leq C\|\bv\|_{\mathrm{L}^{\infty}(0,T;\V)}^r\leq C\|\bv\|_{\mathcal{X}}^r. 
		\end{align*}
		Since $1\leq q<r$, a similar calculation yields 
			\begin{align*}
			\|\widetilde{\mathcal{C}}(\bv)\|_{\mathrm{L}^2(0,T;\H)}\leq  C\|\bv\|_{\mathcal{X}}^q. 
		\end{align*}
	Therefore, we have 
		\begin{align}\label{eqn-map}
		\|\Phi(\bv)\|_{\mathcal{X}}&=\left\|\mathcal{S}\left(\f-\B(\bv)-\beta\mathcal{C}(\bv)-\gamma\widetilde{\mathcal{C}}(\bv)	\right)\right\|_{\mathcal{X}} \leq \left\|		\f-\B(\bv)-\beta\mathcal{C}(\bv)-\gamma\widetilde{\mathcal{C}}(\bv)\right\|_{\mathcal{X}} \nonumber\\&\leq C\left( \|\f\|_{\mathrm{L}^2(0,T;\H)}+\|\bv\|_{\mathcal{X}}^2+\|\bv\|_{\mathcal{X}}^r+\|\bv\|_{\mathcal{X}}^q\right). 
	\end{align}
		Let $R>0$ and define the closed ball of radius $R$ as $$\mathcal{B}_R:=\left\{\bv\in\mathcal{X}_{\mathrm{per}}:\|\bv\|_{\mathcal{X}}\leq R\right\}.$$ 
		Then from \eqref{eqn-map}, we infer 
		\begin{align*}
		\|\Phi(\bv)\|_{\mathcal{X}} \le C\left(	\|\f\|_{\mathrm{L}^2(0,T;\H)}+ R^2 + R^r + R^q	\right).
	\end{align*}
		Let us choose  $R=2C\|\f\|_{\mathrm{L}^2(0,T;\H)}$. If  $\|\f\|_{\mathrm{L}^2(0,T;\H)}$ is sufficiently small so that $C(R^2+R^r+R^q) \le \frac{R}{2},$ then $\Phi(\mathcal{B}_R)\subset \mathcal{B}_R$.

	Let us now show that $\Phi$ is a contraction on $\mathcal{B}_R$.	We write
	\begin{align*}
		\Phi(\bv_1)-\Phi(\bv_2)= \mathcal{S}\left( - \B(\bv_1)+\B(\bv_2) - \beta(\mathcal{C}(\bv_1)-\mathcal{C}(\bv_2)) - \gamma(\widetilde{\mathcal{C}}(\bv_1)-\widetilde{\mathcal{C}}(\bv_2))\right),
	\end{align*}
		so that 
	\begin{align*}
		\|\Phi(\bv_1)-\Phi(\bv_2)\|_{\mathcal{X}}\leq C\|\mathcal{F}(\bv_1)-\mathcal{F}(\bv_2)\|_{\mathrm{L}^2(0,T;\H)},
	\end{align*}
	where 
	$\mathcal{F}(\bv)=\B(\bv)+\beta(\mathcal{C}(\bv)+\gamma\widetilde{\mathcal{C}}(\bv)$. We know that $\B(\bv_1)-\B(\bv_2)=\B(\bv_1,\bv_1-\bv_2)+\B(\bv_1-\bv_2,\bv_2)$ and by using H\"older's and Agmon's inequalities, we find 
	\begin{align*}
	&\int_0^T	\|\B(\bv_1(t))-\B(\bv_2(t))\|_{\H}^2\d t\nonumber\\&\leq 2\int_0^T\|\B(\bv_1(t),\bv_1(t)-\bv_2(t))\|_{\H}^2\d t+2\int_0^T\|\B(\bv_1(t)-\bv_2(t),\bv_2(t))\|_{\H}^2\d t \nonumber\\&\leq 2\int_0^T\left(\|\bv_1(t)\|_{\wi\L^{\infty}}^2+\|\bv_2(t)\|_{\wi\L^{\infty}}^2\right)\|\bv_1(t)-\bv_2(t)\|_{\V}^2\d t
	\nonumber\\&\leq C\sqrt{T}\left[\sup_{t\in[0,T]}\|\bv_1(t)\|_{\V}\left(\int_0^T\|\A\bv_1(t)\|_{\H}^2\d t\right)^{1/2}+\sup_{t\in[0,T]}\|\bv_2(t)\|_{\V}\left(\int_0^T\|\A\bv_2(t)\|_{\H}^2\d t\right)^{1/2}\right]\nonumber\\&\quad\times\sup_{t\in[0,T]}\|\bv_1(t)-\bv_2(t)\|_{\V}^2\nonumber\\&\leq CR^2\|\bv_1-\bv_2\|_{\mathcal{X}}^2,
	\end{align*}
	for all $\bv_1,\bv_2\in \mathcal{B}_R$. Let us now consider $\|\mathcal{C}(\bv_1)-\mathcal{C}(\bv_2)\|_{\mathrm{L}^2(0,T;\H)}$ and estimate it using Taylor's formula and Sobolev's inequality as 
	\begin{align*}
		&\int_0^T	\|\mathcal{C}(\bv_1(t))-\mathcal{C}(\bv_2(t))\|_{\H}^2\d t\nonumber\\&\leq C\int_0^T\left(\|\bv_1(t)\|_{\wi\L^{2r}}^{2(r-1)}+\|\bv_2(t)\|_{\wi\L^{2r}}^{2(r-1)}\right)\|\bv_1(t)-\bv_2(t)\|_{\wi\L^{2r}}^2\d t
		\nonumber\\&\leq C\sup_{t\in[0,T]}\left(\|\bv_1(t)\|_{\V}^{2(r-1)}+\|\bv_2(t)\|_{\V}^{2(r-1)}\right)\sup_{t\in[0,T]}\|\bv_1(t)-\bv_2(t)\|_{\V}^2\nonumber\\&\leq CR^{2(r-1)}\|\bv_1-\bv_2\|_{\mathcal{X}}^2,
	\end{align*}
	for $r\in[1,3]$ and for all $\bv_1,\bv_2\in \mathcal{B}_R$. A similar calculation yields 
	\begin{align*}
		&\int_0^T	\|\widetilde{\mathcal{C}}(\bv_1(t))-\widetilde{\mathcal{C}}(\bv_2(t))\|_{\H}^2\d t\leq CR^{2(q-1)}\|\bv_1-\bv_2\|_{\mathcal{X}}^2,
	\end{align*}
for $1\leq q<r\leq 3$ and for all $\bv_1,\bv_2\in \mathcal{B}_R$. 
		Combining all the above estimates, we obtain
	\begin{align*}
		\|\mathcal{F}(\bv_1)-\mathcal{F}(\bv_2)\|_{\mathrm{L}^2(0,T;\H)}\leq C\left(R+R^{r-1}+R^{q-1}\right)\|\bv_1-\bv_2\|_{\mathcal{X}}. 
	\end{align*}
		Therefore, we have 
	\begin{align*}
		\|\Phi(\bv_1)-\Phi(\bv_2)\|_{\mathcal{X}}\leq C\left(R+R^{r-1}+R^{q-1}\right)\|\bv_1-\bv_2\|_{\mathcal{X}}. 
	\end{align*}
		If $R$ is chosen sufficiently small so that $ C\big(R + R^{r-1} + R^{q-1}\big) < 1,	$ then $\Phi$ is a contraction on $\mathcal{B}_R$. By an application of the Banach fixed point theorem, we infer the existence of a  unique $\bv\in \mathcal{B}_R$ such that $\Phi(\bv)=\bv$. Therefore, $d=3$ and $r\in[1,3]$,  if $\|\f\|_{\mathrm{L}^2(0,T;\H)}$ is sufficiently small, then  there exists a unique  time-periodic strong solution $\bv\in \mathcal{X}$ for the problem \eqref{222}. Since $\bv$ lies in $\mathcal{B}_R$, we have 
		\begin{align*}
			\|\bv\|_{\mathcal{X}}\leq R=2C\|\f\|_{\mathrm{L}^2(0,T;\H)},
		\end{align*}
		which completes the proof of the estimate \eqref{ener-storng}. 
		\end{proof} 
		
		\begin{remark}
		1. Owing to the Sobolev and Gagliardo-Nirenberg inequalities, the fixed point argument used in Theorem \ref{thm-app-2} remains valid for $r\in[1,5]$ in three dimensions and for all $r\in[1,\infty)$ in two dimensions.
		
		2. For the case $d=2$ and $r\in[1,3]$, since  $\f\in\mathrm{L}^2(0,T;\H)$, one may proceed by considering the Galerkin approximated system \eqref{a1} and derive uniform bounds of the form
		$$	\|\boldsymbol{v}_m(t)\|_{\V}^2 + \int_0^t \|\A \boldsymbol{v}_m(s)\|_{\H}^2 \d s + \int_0^t \|\boldsymbol{v}_m'(s)\|_{\H}^2 \d s\leq C\|\f\|_{\mathrm{L}^2(0,T;\H)},$$ for all $t\in[0,T]$. By the Banach-Alaoglu theorem, we may extract subsequences that converge weakly in the corresponding spaces. Moreover, by the Aubin-Lions-Simon  compactness theorem, we obtain strong convergence in 
		$\mathrm{L}^2(0,T;\V).$ Passing to the limit in the Galerkin system \eqref{a1}, we conclude the existence of a strong solution $\bv\in\mathrm{C}([0,T];\V)\cap\mathrm{L}^2(0,T;\D(\A)),$ with $\partial_t\bv\in\mathrm{L}^2(0,T;\H)$  to the problem \eqref{eqn-model}.
	If $\|\f\|_{\mathrm{L}^2(0,T;\H)}$	is chosen sufficiently small, the bilinear term can be controlled, and uniqueness of weak solutions follows similarly as in Step (3) in the proof of Theorem \ref{thm2.11} (see \cite[Section 5]{HMo} for the 2D NSE).
		\end{remark}
		
	\end{appendix}}

		\medskip\noindent
		{\bf Acknowledgements:} Support for M. T. Mohan's research received from the National Board of Higher Mathematics (NBHM), Department of Atomic Energy, Government of India (Project No. 02011/13/2025/NBHM(R.P)/R\&D II/1137). We sincerely thank the reviewers for their careful reading and constructive suggestions, which have significantly improved the quality and clarity of our paper.

			\medskip\noindent	{\bf  Declarations:} 
		
		\noindent 	{\bf  Ethical Approval:}   Not applicable 
		
		\noindent  {\bf   Competing interests: } The authors declare no competing interests. 
		
		\noindent 	{\bf   Authors' contributions: } All authors have contributed equally. 
		

		\noindent 	{\bf   Availability of data and materials: } Not applicable.


\begin{thebibliography}{99}
			
			
			
			
			
			\bibitem{SNA1} S. N. Antontsev and H. B. de Oliveira, Navier-Stokes equations with absorption under slip boundary conditions: existence, uniqueness and extinction in time, Kyoto Conference on the Navier-Stokes Equations and their Applications, Kyoto, \emph{Res. Inst. Math. Sci.}, (2007), 21--41.
			
			\bibitem{SNA}	S.~N. Antontsev and H.~B. de~Oliveira, The Navier-Stokes problem modified by an absorption term, \emph{Appl. Anal.} {\bf 89} (2010), no.~12, 1805--1825
			
			
			
			
			
%
%
%
%
%
%
%
%
%
%
%
%
%
			
			
%
%
%
%
%


 \bibitem{ZCQJ} X. Cai and Q. Jiu, Weak and strong solutions for the incompressible Navier-Stokes equations with damping, \emph{J. Math. Anal. Appl.} {\bf 343} (2008), no.~2, 799--809.

\bibitem{VVC+MIV-02}  V.~V. Chepyzhov and M.~I. Vishik, \emph{Attractors for Equations of Mathematical Physics}, American Mathematical Society Colloquium Publications, 49, Amer. Math. Soc., Providence, RI, 2002.

	\bibitem{PCAM} P. Cherrier and A.~J. Milani, \emph{Linear and Quasi-linear Evolution Equations in Hilbert Spaces}, Graduate Studies in Mathematics, 135, Amer. Math. Soc., Providence, RI, 2012.
			
			
			
			
			
			
			
			
%
%
%
	\bibitem{LCE} L.~C. Evans, \emph{Partial Differential Equations}, second edition, 
	Graduate Studies in Mathematics, 19, Amer. Math. Soc., Providence, RI, 2010.
%
%
%
			
			
			
			
				\bibitem{FKS} R. Farwig, H. Kozono and H. Sohr, An $L^q$-approach to Stokes and Navier-Stokes equations in general domains, \emph{Acta Math.} {\bf 195} (2005), 21--53.
			
			
			
			
			
			\bibitem{CLF} 	C.~L. Fefferman, K.~W. Hajduk and J.~C. Robinson, Simultaneous approximation in Lebesgue and Sobolev norms via eigenspaces, \emph{Proc. Lond. Math. Soc. (3)} {\bf 125} (2022), no.~4, 759--777.
			
			
			
			\bibitem{DFHM} D. Fujiwara and H. Morimoto, An $L\sb{r}$-theorem of the Helmholtz decomposition of vector fields, \emph{J. Fac. Sci. Univ. Tokyo Sect. IA Math.} {\bf 24} (1977), no.~3, 685--700.
			
			
			
			
			
				\bibitem{galdi-1} G.~P. Galdi, An introduction to the Navier-Stokes initial-boundary value problem, in {\it Fundamental directions in mathematical fluid mechanics}, 1--70, Adv. Math. Fluid Mech., Birkh\"auser, Basel, 2000. 
			
			
			\bibitem{galdi} G.~P. Galdi, \emph{An Introduction to the Mathematical Theory of the Navier-Stokes Equations}, second edition,  Springer Monographs in Mathematics, Springer, New York, 2011.
			
			\bibitem{GPAL-06}  G.~P. Galdi and A.~L. Silvestre, Existence of time-periodic solutions to the Navier-Stokes equations around a moving body, \emph{Pacific J. Math.} {\bf 223} (2006), no.~2, 251--267.
			
			\bibitem{GPG-13}  G.~P. Galdi, Existence and uniqueness of time-periodic solutions to the Navier-Stokes equations in the whole plane, \emph{Discrete Contin. Dyn. Syst. Ser. S}  {\bf 6} (2013), no.~5, 1237--1257.
			
			\bibitem{GPG-22} G.~P. Galdi, Existence, uniqueness and asymptotic behavior of regular time-periodic solutions to the Navier-Stokes equations around a moving body: rotational case, \emph{Indiana Univ. Math. J.}, {\bf 71} (2022), no.~6, 2259--2281.
			
			
			
			\bibitem{MTM7} S. Gautam and M.~T. Mohan, On the convective Brinkman-Forchheimer equations, \emph{Dyn. Partial Differ. Equ.} {\bf 22} (2025), no.~3, 191--233.
			
			
		\bibitem{SGKKMTM} 	S. Gautam, K. Kinra and M.~T. Mohan, Feedback stabilization of convective Brinkman-Forchheimer extended Darcy equations, \emph{Appl. Math. Optim.} {\bf 91} (2025), no.~1, Paper No. 25, 75 pp.
			
			
			
			
			
			
			\bibitem{KWH}	K.~W. Hajduk and J.~C. Robinson, Energy equality for the 3D critical convective Brinkman-Forchheimer equations, \emph{J. Differential Equations} {\bf 263} (2017), no.~11, 7141--7161.
			
			
			\bibitem{MHe} M. Hieber et al., $L^r$-Helmholtz-Weyl decomposition for three dimensional exterior domains, \emph{J. Funct. Anal.} {\bf 281} (2021), no.~8, Paper No. 109144, 52 pp.
			
			
				\bibitem{RAJPN-00} R.~A. Johnson, P. Nistri and M.~I. Kamenskiui, On the existence of periodic solutions of the Navier-Stokes equations in a thin domain using the topological degree, \emph{J. Dynam. Differential Equations}, {\bf 12} (2000), no.~4, 681--712.
			
			
			\bibitem{KT2}  V. K. Kalantarov and S. Zelik, Smooth attractors for the Brinkman-Forchheimer equations with fast growing nonlinearities, \emph{Commun. Pure Appl. Anal.} {\bf 11}	(2012) 2037--2054.
			
			
		\bibitem{HK-97}	H. Kato, Existence of periodic solutions of the Navier-Stokes equations, \emph{J. Math. Anal. Appl.} {\bf 208} (1997), no.~1, 141--157.
			
			
			 \bibitem{HKTY}   H. Kozono and T. Yanagisawa, 
			$L^r$-variational inequality for vector fields and the Helmholtz-Weyl decomposition in bounded domains,
			\emph{Indiana Univ. Math. J.} {\bf  58} (2009),  no.~4, 1853--1920.
			
		\bibitem{HKYMRT-14}	H. Kozono, Y. Mashiko and R. Takada, Existence of periodic solutions and their asymptotic stability to the Navier-Stokes equations with the Coriolis force, \emph{J. Evol. Equ.}, {\bf 14} (2014), no.~3, 565--601.
			
		\bibitem{MK-14}	M. Kyed, The existence and regularity of time-periodic solutions to the three-dimensional Navier-Stokes equations in the whole space, \emph{Nonlinearity}, {\bf 27} (2014), no.~12, 2909--2935
			
			
			
			
			
			
			
			
			
			\bibitem{DL-90} D. Lauerov\'a, The Rothe method and time periodic solutions to the Navier-Stokes equations and equations of magnetohydrodynamics, \emph{Apl. Mat.}, {\bf 35} (1990), no.~2, 89--98.
			
			
			
			
			
			
		\bibitem{lions-1}	J.-L. Lions, \emph{Quelques M\'ethodes de R\'esolution des Probl\`emes aux Limites Non lin\'eaires}, Dunod, Paris, 1969 Gauthier-Villars, Paris, 1969
			
			
		\bibitem{JLL+EM-I-72}	J.-L. Lions and E. Magenes, {\it Non-homogeneous Boundary Value Problems and Applications. Vol. I}, Springer, New York-Heidelberg, 1972.
			
			
			
			
			%
			%
			%
			%
			%
			%
			
			
			
		\bibitem{PMMP-99}	P. Maremonti and M. Padula,
			Existence, uniqueness and attainability of periodic solutions of the Navier-Stokes equations in exterior domains, \emph{J. Math. Sci.} {\bf 93}  (1999), no.~5,  719--746.
			
			\bibitem{PAM}	P. A. Markowich, E. S. Titi and S. Trabelsi,	Continuous data assimilation for the three-dimensional Brinkman-Forchheimer-extended Darcy model, \emph{Nonlinearity} {\bf 29} (2016), no.~4, 1292--1328.
			
			
			
			
			
			\bibitem{MTM8}  M. T. Mohan, 
			Well-posedness and asymptotic behavior of stochastic convective Brinkman-Forchheimer equations perturbed by pure jump noise,
			\emph{Stoch. Partial Differ. Equ. Anal. Comput.} {\bf 10}  (2022), no.~2, 614--690. 
			
%
%
			
			
			\bibitem{HMo}	H. Morimoto, 	Time periodic Navier-Stokes flow with nonhomogeneous boundary condition, \emph{J. Math. Sci. Univ. Tokyo} {\bf 16} (2009), no.~1, 113--123.
			
			
			
		\bibitem{TN-25}	T. Nakatsuka, Existence of time-periodic strong solutions to the Navier-Stokes equation in the whole space, \emph{J. Math. Anal. Appl.},  {\bf 543} (2025), no.~2, Paper No. 128991, 24 pp.
			
			
		\bibitem{ANFM-25}	A. Novruzi and F. Mezatio, Existence and uniqueness of a time-periodic strong solution to incompressible Navier-Stokes equations in a time-periodic moving domain, describing the blood flow in an artificial heart, \emph{J. Math. Anal. Appl.}, {\bf 548} (2025), no.~2, Paper No. 129410, 28 pp.
			
			
		\bibitem{GPr}	G. Prouse, Soluzioni periodiche dell'equazione di Navier-Stokes, \emph{Atti Accad. Naz. Lincei Rend. Cl. Sci. Fis. Mat. Nat. (8)} {\bf 35} (1963), 443--447
			
			
			
			
			
			
		
			
			
			
			\bibitem{RS-94-1} R. Salvi, On the existence of periodic weak solutions of Navier-Stokes equations in regions with periodically moving boundaries, \emph{Acta Appl. Math.},  {\bf 37} (1994), no.~1-2, 169--179.
			
			\bibitem{RS-94} R. Salvi, On the existence of periodic weak solutions on the Navier-Stokes equations in exterior regions with periodically moving boundaries, in {\it Navier-Stokes Equations and Related Nonlinear Problems (Funchal, 1994)}, 63--73, Plenum, New York.
			
			
			
			  
				\bibitem{CSHS}	C. Simader and H. Sohr,  A new approach to the Helmholtz decomposition and the Neumann problem in $L^q$-spaces for bounded and exterior domains, \emph{Mathematical Problems Relating to the Navier-Stokes Equation}, 1-35,	Ser. Adv. Math. Appl. Sci., 11, World Sci. Publ., River Edge, NJ, 1992.
			
			\bibitem{SJ}  J. Simon, Compact sets in the space $L^p(0,T;B)$, \emph{Ann. Mat. Pura Appl. (4)} {\bf 146} (1987), 65--96.
			
			
%
%

\bibitem{HSo}  H. Sohr, \emph{The Navier-Stokes Equations}, Birkh\"auser Advanced Texts: Basler Lehrb\"ucher, Birkh\"auser, Basel, 2001
			
			
			
			
			
			
			\bibitem{Te} R.~M. Temam, \emph{Navier-Stokes Equations}, third edition, 
			Studies in Mathematics and its Applications, 2, North-Holland, Amsterdam, 1984.
			
			
			
			
			
			
			
			
			
			\bibitem{ZZXW}	Z. Zhang, X. Wu and M. Lu, On the uniqueness of strong solution to the incompressible Navier-Stokes equations with damping, \emph{J. Math. Anal. Appl.} {\bf 377} (2011), no.~1, 414--419.
			
			
			
			
		\end{thebibliography}
	\end{document}